\documentclass[11pt]{amsart}
\pdfoutput=1

\usepackage{mathptmx}
\usepackage{amssymb}
\usepackage{tikz}
\usepackage{booktabs}
\usepackage{multicol}
\usepackage{placeins}
\usepackage{caption} 
\usepackage{subcaption} 
\usepackage{graphicx}
\usepackage{mathdots}
\usepackage{mathrsfs}
\usepackage{txfonts}
\usepackage{amsopn}
\usepackage{amsfonts}

\usetikzlibrary{arrows,automata, positioning, calc, matrix}
\newtheorem{thm}{Theorem}[section]

\newtheorem{lem}[thm]{Lemma}

\newtheorem{defn}[thm]{Definition}

\definecolor{Red}{rgb}{1.0, 0, 0}
\definecolor{Green}{rgb}{0, 1.0, 0}
\makeindex

\begin{document}
		
\title[Second Minimal Odd Periodic Orbits]{Classification of the Second Minimal Odd Periodic Orbits in the Sharkovskii Ordering}\thanks{Department of Mathematics, Florida Institute of Technology, Melbourne, FL 32901}

	\author[U.G.Abdulla]{Ugur G. Abdulla}
	\address{Department of Mathematics, Florida Institute of Technology, Melbourne, FL 32901}
	\email{abdulla@fit.edu}	
	
	\author[R.U.Abdulla]{Rashad U. Abdulla}	
		
	\author[M.U.Abdulla]{Muhammad U. Abdulla}		
				
	\author[N.H.Iqbal]{Naveed H. Iqbal}
	
	\begin{abstract}
	This paper presents full classification of second minimal odd periodic orbits of a continuous endomorphisms on the real line. A $(2k+1)$-periodic orbit $\{\beta_{1}<\beta_{2}<\cdots <\beta_{2k+1} \}$, ($k\geq 3$) is called second minimal  for the map $f$, if $2k-1$ is a minimal period of $f|_{[\beta_1,\beta_{2k+1}]}$ in the Sharkovski ordering. We prove that second minimal odd orbits either have a Stefan structure like minimal odd orbits, or have one of the $4k-3$ types, each characterized with unique cyclic permutation and directed graph of transitions  with accuracy up to inverses. 
\end{abstract}
	\maketitle

	\section{Introduction and Main Result}
\label{sec:introduction}
Let $f: I\rightarrow I$ be a continuous endomorphism, and $I$ be a non-degenerate interval on the real line. Let $f^n: I \rightarrow I$ be an $n$-th iteration of $f$. A point $c\in I$ is called a periodic point of $f$ with period $m$ if $f^m(c)=c$, $f^k(c)\neq c$ for $1\leq k <m$. The set of $m$ distinct points
\[ c, f(c), \cdots , f^{m-1}(c) \]
is called the orbit of $c$, or briefly $m$-orbit or periodic $m$-cycle. In his celebrated paper \cite{sharkovsky64}, Sharkovski discovered a law on the coexistence of periodic orbits of continuous endomorphisms on the real line.

\begin{thm}
	\label{thm:sharkovskii}
	\cite{sharkovsky64} Let the positive integers be totally ordered in the following way:
	
	\begin{equation}\label{sharkovskiordering}
		1\triangleleft 2\triangleleft 2^{2} \triangleleft 2^{3}\triangleleft\dots\triangleleft 2^{2}\cdot 5\triangleleft2^{2}\cdot 3\triangleleft\dots\triangleleft 2\cdot 5\triangleleft 2\cdot 3 \triangleleft\dots\triangleleft 9\triangleleft 7\triangleleft 5\triangleleft 3.
	\end{equation}

	\noindent If a continuous endomorphism, $f:I\rightarrow I$, has a cycle of period $n$ and $m\triangleleft n$, then $f$ also has a periodic orbit of period $m$.
\end{thm}

This result played a fundamental role in the development of the theory of discrete dynamical systems. 
Following the standard approach (\cite{Block1979,block92,alseda2000}), we characterize each periodic orbit with cyclic permutations and directed graphs of transitions or {\it digraphs}. Consider the $m$-orbit:

\[ {\bf B}=\{\beta_{1}<\beta_{2}<\cdots <\beta_{m} \} \]
\begin{defn}
	If $f(\beta_i)=\beta_{s_i}$ for $1\le s_i \le m$, with $i=1,2,...,m$, then ${\bf B}$ is associated with cyclic permutation
	\[\pi=
	\begin{pmatrix}
	1&2&\dots&m\\
	s_1&s_2&\dots&s_m
	\end{pmatrix}
	\]
\end{defn}

In the sequel $<a,b>$ means either $[a,b]$ or $[b,a]$.

\begin{defn}
	Let $J_i=[\beta_i, \beta_{i+1}]$. The digraph of $m$-orbit is a directed graph of transitions with vertices $J_1,J_2,\cdots,J_{m-1}$
	and oriented edges $J_i \rightarrow J_s$ if $J_s \subset \  <f(\beta_i), f(\beta_{i+1})>$.
\end{defn}

\begin{defn}
         The inverse digraph of $m$-orbit is obtained from the digraph of $m$-orbit by replacing each $J_i$ with $J_{m-i}$. 
\end{defn}
The inverse of the digraph associated with the cyclic permutation $\pi$ is a digraph associated with the cyclic permutation $\omega\circ \pi\circ \omega$, where $\omega$ be the order reversing permutation:
	
	\[
	\omega =\begin{pmatrix}
	1 & 2 & \dots & m-1 & m\\
	m & m-1 & \dots & 2 & 1
	\end{pmatrix}
	\]	
\begin{defn}
	A continuous function $P_{f}: \left [ \beta_{1}, \beta_{m} \right ]\rightarrow \left [ \beta_{1}, \beta_{m} \right ]$ is called the $P$-linearization of $f$ if $P_{f}\left ( \beta_{i} \right ) = f\left ( \beta_{i} \right )$ and $P$ is a linear function in each interval $J_{i}$
\end{defn}

\begin{defn}
	The arrangement of the minimums and maximums of the map $P_{f}$ in the open interval $\left ( \beta_{1}, \beta_{m} \right )$ will be called the topological structure of the periodic orbit.
\end{defn}

The proof of Sharkovski's theorem significantly uses the concept of {\it minimal orbit}.

\begin{defn}
	$m$-orbit of $f$ is called minimal if $m$ is the minimal period of $f|_{[\beta_1,\beta_m]}$ in the Sharkovski ordering.
\end{defn}

\definecolor{Red}{rgb}{1.0, 0, 0}
\begin{defn}
	Digraph of the $m$-orbit contains the red edge 
	$J_i{\color{Red} \rightarrow}J_s$ if $J_s=<f(\beta_i), f(\beta_{i+1})>$. 
\end{defn}
The structure of the minimal orbits is well understood \cite{Stefan1977, alseda1984, Block1986, block92, alseda2000,abdulla2013}. Minimal odd orbits are called Stefan orbits, due to the following characterization:
\begin{thm}
	\label{thm:stefan}
		\cite{Stefan1977} The digraph of a minimal $2k+1$-orbit, $k\geq 1$, has the unique structure given in Fig.~\ref{fig:minOddDigraph} and cyclic permutation \eqref{eq:stefanorbit} up to an inverse. 
		
		\begin{equation}
			\begin{pmatrix}
			1 & 2 & 3 & \cdots & k & k+1 & k+2 & k+3 & \cdots & 2k & 2k+1 \\ 
			k+1 & 2k+1 & 2k & \cdots & k+3 & k+2 & k & k-1 & \cdots & 2 & 1
			\end{pmatrix}			
		\label{eq:stefanorbit}
		\end{equation}
		
		\begin{figure}[htpb]
			\centering						
			\resizebox{8cm}{!}{\begin{tikzpicture}[
            > = stealth, 
            shorten > = 1pt, 
            auto,
            node distance = 1.5cm, 
            semithick 
        ]
        
				\node[draw=none,fill=none] (jk2) {$J_{k+2}$};
        \node[draw=none,fill=none] (jk) [below of=jk2] {$J_{k}$};
				\node[draw=none,fill=none] (jk1) [left=1cm of {$(jk2)!0.5!(jk)$}] {$J_{k+1}$};				
        \node[draw=none,fill=none] (jk3) [right of=jk2] {$J_{k+3}$};
        \node[draw=none,fill=none] (jkm1) [below of=jk3] {$J_{k-1}$};
        \node[draw=none,fill=none] (dots1) [right of=jk3] {$\dots$};
        \node[draw=none,fill=none] (dots2) [right of=jkm1] {$\dots$};
        \node[draw=none,fill=none] (j2km1) [right of=dots1] {$J_{4}$};
        \node[draw=none,fill=none] (j3) [right of=dots2] {$J_{3}$};
        \node[draw=none,fill=none] (j2k) [right of=j2km1] {$J_{2k}$};
				\node[draw=none,fill=none] (j2) [below of=j2k] {$J_{2}$};
        \node[draw=none,fill=none] (j1) [right=1cm of {$(j2k)!0.5!(j2)$}] {$J_{1}$};        

				\path[->] (jk1) edge[loop left] node{} (jk1);
        \path[->] (jk1) edge node{} (jk);
        \path[->] (jk) edge[red] node{} (jk2);
        \path[->] (jk2) edge[red] node{} (jkm1);
        \path[->] (jkm1) edge[red] node{} (jk3);
				\path[dashed,->] (jk3) edge[red] node{} (dots2);
        \path[dashed,->] (dots2) edge[red] node{} (dots1);
        \path[dashed,->] (dots1) edge[red] node{}  (j3);
        \path[->] (j3) edge[red] node{}  (j2km1);
        \path[->] (j2km1) edge[red] node{}  (j2);
        \path[->] (j2) edge[red] node{}  (j2k);
				\path[->] (j2k) edge[red] node{}  (j1);
				\draw [<-] (-1.6,-0.2) -- (-1.6, 0.2);
				\draw (-1.6,0.2) arc (180:90:8mm) (-0.8,1) -- (6.5,1) (6.5,1) arc (90:0:8mm) (7.3,0.2) -- (7.3, -0.2) -- cycle;
				\draw [<-] (0,0.4) -- (0,1);
				\draw [<-] (1.5,0.4) -- (1.5,1);
				\draw [dashed, <-] (3,0.4) -- (3,1);
				\draw [<-] (4.5,0.4) -- (4.5,1);
				\draw [<-] (6,0.4) -- (6,1);
			\end{tikzpicture}} 		
		\caption{Digraph of Minimal Odd Orbit}
		\label{fig:minOddDigraph}
	\end{figure}
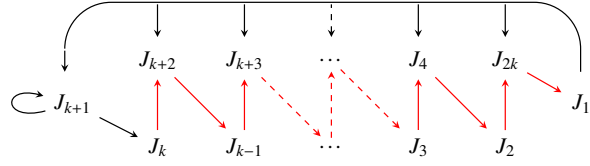	
\end{thm}

The main goal of this paper is the characterization of second minimal odd orbits.
\begin{defn} 
	\label{thm:secondmin}
	A $(2k+1)$-orbit, $k\geq 3$ is called second minimal if $2k-1$ is the minimal period of $f|_{[\beta_1,\beta_{2k+1}]}$ in the Sharkovski ordering. 
\end{defn}
To achieve the full characterization of the second minimal odd orbits, in the next definition we introduce a new notion of {\it simplicity} of odd periodic orbits.
\noindent Let $B(i,j)$, $1\leq i\leq j\leq 2k+1$ be subsets of a $(2k+1)$-orbit defined as

\begin{equation}
	B(i,j) = \left \{ \beta_{k}\in B: i\leq k\leq j \right \}
\label{eq:defnBij}
\end{equation}

\begin{defn}\label{simpleoddorbits}
	A $(2k+1)$-orbit is called simple if either
	\begin{enumerate}
		\item $B(k+2, 2k+1)$ is mapped to $B(1, k+1)$; and $B(1, k+1)$ is mapped to $B(k+2, 2k+1)$ except one point; or \label{en:simpletypea}
		\item $B(1, k)$ is mapped to $B(k+1, 2k+1)$; and $B(k+1, 2k+1)$ is mapped to $B(1, k)$ except one point. \label{en:simpletypeb}
	\end{enumerate}
	
	We say a simple $(2k+1)$-orbit is of type $+$ (resp. type $-$) if (\ref{en:simpletypea}) (resp. (\ref{en:simpletypeb})) is satisfied.
\end{defn}

\noindent First of all note that the Stefan orbits or minimal odd orbits are simple according to Definition~\ref{simpleoddorbits}. Our first main result reads: 
\begin{thm}
	Second minimal $(2k+1)$-orbits, $k\geq 3$, are simple. 
\label{2ndminimalissimple}
\end{thm}
To pursue a full classification of the second minimal odd orbits, first note that second minimal odd orbits may have a Stefan structure identified in Theorem~\ref{thm:stefan}. Indeed, consider a map which is $P$-linearization of the minimal $2k+1$-orbit. It has a unique fixed point which is interior point of one of the two middle intervals. We can replace linear function in the small neighborhood $F$ of the fixed point with $P$-linearization of the minimal $2k-1$-orbit, join this function continuously with the original map outside of the small neighborhood of size twice larger than $F$. Moreover, we can choose the size of $F$ so small that the digraph of the $2k+1$-orbit is not changed and still has a Stefan structure. Obviously, $2k+1$-orbit is second minimal with respect to the new map, although its Stefan structure is unchanged. Therefore, to complete the full classification it remains to clarify the structure of all second minimal odd orbits with non-Stefan structure. Our main classification result reads:
\begin{thm}	
	Simple positive type second minimal $2k+1$-orbits are either Stefan orbits, or have one of the $4k-3$ types, each with unique digraph and cyclic permutation. Their inverses represent all second minimal $(2k+1)$-orbits of simple negative type. The topological structure of all $4k-3$ simple positive types of second minimal $(2k+1)$-orbits with non-Stefan structure is presented in Table~\ref{tab:topstructintro}. The topological structure of their inverses is obtained by replacing ``max'' and ``min'' with each other respectively. The $P$-linearization of each of the $4k-3$ types (and their inverses) presents an example of a continuous map with a second minimal $(2k+1)$-orbit.
	
	\begin{table}%
	\centering
	\begin{tabular}{r|c}\toprule
		Topological Structure & Count  \\ \midrule
		max & $1$ \\
		min-max & $1$ \\
		min-max-min & $1$ \\
		max-min & $2$ \\
		max-min-max & $2k-3$ \\
		max-min-max-min-max & $2k-5$\\ \bottomrule
	\end{tabular}
	\caption{Topological structure of all $4k-3$ second minimal $(2k+1)$-orbits of simple positive type with non-Stefan structure}
	\label{tab:topstructintro}
	\end{table}
	
	\label{thm:secondminodd}				
\end{thm}

Theorems~\ref{2ndminimalissimple} \& \ref{thm:secondminodd} in the particular case $k=3$ was proved in \cite{abdulla2016}. Proof of Theorems~\ref{2ndminimalissimple} \& \ref{thm:secondminodd} is constructive, and provides explicit description of all types of second minimal odd orbits in terms of cyclic permutations and digraphs.
It should be pointed out that our main results can be formulated in the framework of formalized combinatorial dynamics, where without any reference to orbits, and associated maps, the objects are permutations (or patterns), and the main problem is to identify forcing relation between various patterns (see \cite{block92, alseda2000}). 

The structure of the remainder of the paper is as follows: In Section 2, we recall some preliminary facts. Theorems~\ref{2ndminimalissimple} \& \ref{thm:secondminodd} are proved in Section 3.       
	\section{Preliminary Results}
\label{sec:prelim}

\begin{lem}\label{preliminarylemma}
	The digraph of an $m$-orbit, ${\bf B}=\left\{\beta_{1}<\beta_{2}<\cdots <\beta_{m} \right\}$, $m>2$, possesses the following properties \cite{block92}:
	\begin{enumerate}
		\item The digraph contains a loop: $\exists r_{\ast}$ such that $J_{r_{\ast}}\rightarrow J_{r_{\ast}}$.
		\item $\forall r$, $\exists {r}'$ and ${r}''$ such that $J_{{r}'}\rightarrow J_{r} \rightarrow J_{{r}''}$; moreover, it is always possible to choose ${r}'\neq r$ unless $m$ is even and $r=m/2$, and it is always possible to choose ${r}''\neq r$ unless $m=2$.
		\item If $\left [ {\beta}', {\beta}''\right ]\neq\left [ \beta_{1}, \beta_{m}\right ]$, ${\beta}', {\beta}''\in {\bf B}$, then $\exists J_{{r}'}\subset\left [ {\beta}', {\beta}''\right ]$ and $\exists J_{{r}'}\nsubseteq\left [ {\beta}', {\beta}''\right ]$ such that $J_{{r}'}\rightarrow J_{{r}''}$.
		\item The digraph of a cycle with period $m>2$ contains a subgraph $J_{r_{\ast}}\rightarrow\cdots J_{r}$ for any $1\leq r\leq m-1$.
	\end{enumerate}
\end{lem}

\begin{defn} 
	A cycle in a digraph is said to be primitive if it does not consist entirely of a cycle of smaller length described several times.
\end{defn}

\begin{lem}[Straffin]
	\label{thm:straffin}
	\cite{Straffin1978, block92} If $f$ has a periodic point of period $n>1$ and its associated digraph contains a primitive cycle $J_{0}\rightarrow J_{1}\rightarrow\dots\rightarrow J_{m-1}\rightarrow J_{0}$ of length $m$, then f has a periodic point $y$ of period $m$ such that $f^{k}(y) \in J_{k}, (0\leq k < m)$.
\end{lem}

\begin{lem}[Converse Straffin]
	\cite{block92} Let $f$ have a periodic point of period $n>1$ with digraph $\mathbb{D}$. Suppose $f$ is strictly monotonic on each subinterval $J_{i} = [\beta_{i} \beta_{i+1}]$ for $1 \leq i \leq n-1$. If $f$ has an orbit of period $m$ in the open interval $(\beta_{1}, \beta_{n})$ then either $\mathbb{D}$ contains a primitive cycle of length $m$, or $m$ is even and $\mathbb{D}$ contains a primitive cycle of length $m/2$.
	\label{thm:convStraffin}
\end{lem}

	\section{Proofs of Theorems~\ref{2ndminimalissimple} \& \ref{thm:secondminodd} }
\label{sec:proofs}


	Let $f:I\rightarrow I$ be a continuous endomorphism that has a $2k+1$-orbit ($k\geq 4$) which is second minimal. Let $B = \left \{ \beta_{1}< \beta_{2} < \cdots < \beta_{2k+1}  \right \}$ be the ordered elements of this orbit; Let $r_{\ast} = \max\left \{ i\mid f(\beta_{i}) > \beta_{i} \right \}$. Such an $r_{\ast}$ exists since $f(\beta_{1}) > \beta_{1}$ and $f(\beta_{2k+1}) < \beta_{2k+1}$. So, $J_{r_{\ast}}\rightarrow J_{r_{\ast}}$; Let $B^{-} = \left \{ \beta\in B\mid \beta \leq \beta_{r_{\ast}} \right \}, B^{+} = \left \{ \beta\in B\mid \beta > \beta_{r_{\ast}} \right \}$; We have $\left | B^{-} \right | + \left | B^{+} \right | = 2k+1$ and hence $\left | B^{-} \right |\neq\left | B^{+} \right |$. Assume, without loss of generality, $\left | B^{-} \right | > \left | B^{+} \right |$. Let $r=\max\left \{ i < r_{\ast}\mid f(\beta_{i})\leq\beta_{r_{\ast}} \right \}$. We have $f(\beta_{r})\leq \beta_{r_{\ast}}$,  $f(\beta_{r+1})> \beta_{r_{\ast}}$, and hence $J_{r}\rightarrow J_{r_{\ast}}$. From Lemma \ref{preliminarylemma} it follows the existence of the subgraph

\begin{equation}
	\circlearrowright J_{r_{\ast}} \rightarrow \cdot \cdot \cdot \rightarrow J_{r} \rightarrow J_{r_{\ast}}
\label{eq:secondminfc_orig}
\end{equation}
Assume that \eqref{eq:secondminfc_orig} presents the shortest path. Since there are $2k$ intervals, its length is at most $2k+1$ and at least $2k-1$. Indeed, if its length is $2k-2$ or less, then Lemma \ref{thm:straffin} implies the existence of an odd periodic orbit of period $2k-3$ or less. Let us change the indices of intervals in \eqref{eq:secondminfc_orig} successfully as $r_{\ast}=r_{1}, \cdots, r=r_{m}$ and write path \eqref{eq:secondminfc_orig} as

\begin{equation}
	\circlearrowright J_{r_{1}} \rightarrow \cdot \cdot \cdot \rightarrow J_{r_{m}} \rightarrow J_{r_{1}}
\label{eq:secondminfc_index}
\end{equation}

\noindent where $m=2k-2$, $2k-1$, or $2k$; For simplicity we are going to use the notation $i$ for $\beta_{i}$. In the sequel the notation $\begin{matrix}a\\ b \end{matrix}$ in the second row of the cyclic permutation means that either of the entries $a$ or $b$ are valid choices for the image of the node in the same column of the first row; $J_{r_{i}}\rightarrow [a,b]$ means $f(r_{i}) = a$ and $f(r_{i}+1) = b$, the notation $\langle J_{r},J_{s} \rangle$ means the union of $J_{r}$, $J_{s}$, and all the intervals between them. Note that if $J_{r_{i}}\rightarrow J_{r_{j}}$ and $J_{r_{i}}\rightarrow J_{r_{k}}$ then $J_{r_{i}}\rightarrow\left\langle J_{r_{j}}, J_{r_{k}} \right\rangle$. 

Since \eqref{eq:secondminfc_index} is the shortest path we have

\begin{subequations}
	\begin{align}
	J_{r_{i}}& \not\rightarrow J_{r_{j}} \,\,\mathrm{for}\,\, j > i+1, 1\leq i\leq m-2;\label{seq:rules3}\\
	J_{r_{i}}& \not\rightarrow J_{r_{1}}\,\,\mathrm{for}\,\, 2\leq i \leq m-1;\label{seq:rules4}
	\end{align}	
	\label{eq:2minrulesForward}
\end{subequations}

\noindent we also have

\begin{subequations}
	\begin{align}
	J_{r_{i}}& \not\rightarrow J_{r_{j}}\,\,\mathrm{for}\,\, 1 < j < i, 4\leq i \leq m, i-j\,\,\mathrm{even};\label{seq:rules5}
	\end{align}	
	\label{eq:2minrulesBackward}
\end{subequations} 

\noindent unless $i=m=2k$, $j=2$. Indeed, otherwise according to Lemma~\ref{thm:straffin} an odd orbit of length less than $2k-1$ must exist. From \eqref{eq:2minrulesForward} and \eqref{eq:2minrulesBackward} we can infer the relative position of the intervals to be either

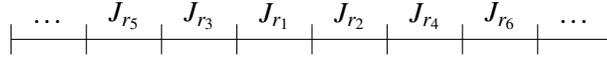
\begin{figure}[htb]
	\centering
	\begin{tikzpicture}
		\draw (1,0)--(9.0,0);
		\foreach \num/\idx in {1,2,...,9}
			{
			\draw (\num,0.2)--(\num,-0.2); 
			}
		\foreach \label/\loc in {\cdots/1.5, J_{r_{5}}/2.5, J_{r_{3}}/3.5, J_{r_{1}}/4.5, J_{r_{2}}/5.5, J_{r_{4}}/6.5, J_{r_{6}}/7.5, \cdots/8.5}
			{
			\node[above] at (\loc, 0) {$\label$};
			}
	\end{tikzpicture}
	\caption{Relative positions of intervals in the sub-graph of length $m$ when $J_{r_{2}}$ to the right of $J_{r_{1}}$}
	\label{fig:secminint1}
\end{figure}

\noindent or

\begin{figure}[htb]
	\centering
	\begin{tikzpicture}
		\draw (1,0)--(9.0,0);
		\foreach \num/\idx in {1,2,...,9}
			{
			\draw (\num,0.2)--(\num,-0.2); 
			}
		\foreach \label/\loc in {\cdots/1.5, J_{r_{6}}/2.5, J_{r_{4}}/3.5, J_{r_{2}}/4.5, J_{r_{1}}/5.5, J_{r_{3}}/6.5, J_{r_{5}}/7.5, \cdots/8.5}
			{
			\node[above] at (\loc, 0) {$\label$};
			}
	\end{tikzpicture}
	\caption{Relative positions of intervals in the sub-graph of length $m$ when $J_{r_{2}}$ to the left of $J_{r_{1}}$}
	\label{fig:secminint2}
\end{figure}
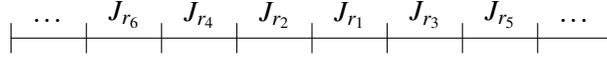

If $m=2k$ then the path \eqref{eq:secondminfc_index} contains all $2k$ intervals. From the proof of Theorem \ref{thm:stefan}  (for example, see Proposition 8 in \cite{block92}) it follows that the corresponding periodic orbit is Stefan orbit. 

\begin{lem}
	The case when $m=2k-1$ produces exactly $2$ second minimal $(2k+1)$-orbits. Cyclic permutations are given in \eqref{eq:len2kcase11cyc} and \eqref{eq:len2kcase22cyc},  and the corresponding digraphs are presented in Fig.~\ref{fig:len2kcase11} and Fig.~\ref{fig:len2kcase22} respectively.

	\small{
	\begin{equation}
		\left(\begin{array}{cccccccccc}
			1   & 2 & \cdots & k-2 & k-1 & k   & k+1 & k+2 & \cdots \\ 
			k+1 & 2k+1 & \cdots & k+5 & k+2 & k+4 & k+3 & k   & \cdots \\
		\end{array} \right)					
		\label{eq:len2kcase11cyc}
	\end{equation}}
	\small{
	\begin{equation}		
			\left(\begin{array}{cccccccccc}
				1 & 2 &\cdots & k-2 & k-1 & k   & k+1 & k+2 & k+3 & \cdots \\ 
				k & 2k+1 & \cdots & k+5 & k+4 & k+2 & k+3 & k+1 & k-1 & \cdots \\
			\end{array} \right)					
		\label{eq:len2kcase22cyc}
	\end{equation}}
	
	\input{digraphs/digraphLen2km1-1.tex}
	
	\input{digraphs/digraphLen2km1-2.tex}
	
	\label{lem:m2km1}
\end{lem}
\begin{proof}
	When the length of the path \eqref{eq:secondminfc_index} is $2k-1$ we have exactly one interval, call it $\tilde{J}$, missing. Since $J_{r_{2k-1}}\rightarrow J_{r_{1}}$, $J_{r_{2k-1}}\not\rightarrow J_{r_{i}}$ for $i>1$ odd, one of the endpoints of $J_{r_{2k-1}}$ must be mapped to some element of the orbit which separates $J_{r_{1}}$ and $J_{r_{3}}$. Since $\circlearrowright J_{r_{1}}\rightarrow J_{r_{2}}$, but $J_{r_{1}}\not\rightarrow J_{r_{3}}$ it follows that the endpoint of $J_{r_{1}}$ which separates $J_{r_{1}}$ and $J_{r_{2}}$, must be mapped to the element of the  orbit which separates $J_{r_{1}}$ and $J_{r_{3}}$. Therefore, unless there is an interval between $J_{r_{1}}$ and $J_{r_{3}}$, the element of the orbit which separates $J_{r_{1}}$ and $J_{r_{3}}$ will be an image of two distinct elements of the orbit. Hence, $\tilde{J}$ is between $J_{r_{1}}$ and $J_{r_{3}}$, and the distribution of the intervals is as in Fig.~\ref{fig:set3in2k} or Fig.~\ref{fig:set3in2k} reflected about the center point $k+1$. Note that the distribution described in Fig.~\ref{fig:set3in2k} is relevant due to our assumption $\left | B^{-} \right | > \left | B^{+} \right |$. The other case will provide the associated inverse digraph with $\left | B^{+} \right | > \left | B^{-} \right |$. Hence, the structure is as it is described in Fig.~\ref{fig:set3in2k}.

	\begin{figure}[htb]
		\centering
		\begin{tikzpicture}
			\draw (1,0)--(11.0,0);
			\foreach \num/\idx in {1,2,...,11}
				{
				\draw (\num,0.2)--(\num,-0.2); 
				}
			\foreach \label/\loc in {J_{r_{2k-1}}/1.5, J_{r_{2k-3}}/2.5, \cdots/3.5, J_{r_{3}}/4.5, \tilde{J}/5.5, J_{r_{1}}/6.5, J_{r_{2}}/7.5, \cdots/8.5, J_{r_{2k-4}}/9.5, J_{r_{2k-2}}/10.5}
				{
				\node[above] at (\loc, 0) {$\label$};
				}
			\foreach \label/\loc in {1/1, 2/2, 3/3, k-1/4, k/5, k+1/6, k+2/7, 2k-1/9, 2k/10, 2k+1/11}
				{
				\node[below] at (\loc, -0.2) {\small{$\label$}};
				}		
		\end{tikzpicture}
		\caption{Complete interval for case when $m=2k-1$ with $\left | B^{-} \right | > \left | B^{+} \right |$}
		\label{fig:set3in2k}
	\end{figure}
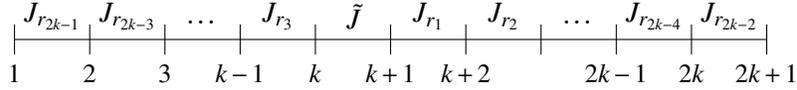

	\noindent This implies the following cyclic permutation. Note the possible alternation of images of elements $1$, $k+2$, $k-1$, and $k$.

	\small{
	\begin{equation}
		\left(\begin{array}{ccccccccc}
			1 & \cdots & k-1 & k & k+1 & k+2 & k+3 & k+4 & \cdots \\ 
			\begin{matrix} k+1\\ k \end{matrix} & \cdots & \langle k+2, & k+4\rangle & k+3 & \begin{matrix} k\\ k+1 \end{matrix} & k-1 & k-2 &\cdots\\
		\end{array} \right)		
	\end{equation}
}
	\begin{enumerate}
		\item Case $(1)$: $f(k-1)=k+2 \Rightarrow f(k)=k+4$
		\item Case $(1)_{1}$: $f(1)=k+1 \Rightarrow f(k+2)=k$; This produces a simple positive type $(2k+1)$-orbit given in \eqref{eq:len2kcase11cyc} and Fig.~\ref{fig:len2kcase11} with topological structure max-min-max. Next we analyze the digraph to show that there are no primitive cycles of even length $\leq 2k-2$, which would imply by Straffin's lemma an existence of odd periodic orbit of length $\leq 2k-3$. From Lemms~\ref{thm:convStraffin} it then follows that the the $P$-linearization of the orbit \eqref{eq:len2kcase11cyc} presents an example of continuous map with second minimal $(2k+1)$-orbit. 
We split the analysis into two cases:
		
			\begin{enumerate}
				\item Consider primitive cycles that contain $J_{1}$. Without loss of generality choose $J_1$ as the starting vertex. First assume that cycle doesn't start with chain $J_1\rightarrow J_{k+1}$. Since  $J_{2k+1-i}{\color{Red}\rightarrow}J_{i},\ i=1,...,k-1,$ any such cycle can be formed only by adding on to the starting vertex $J_1$ pairs $( J_{2k+1-i}, J_{i}), i=1,...,k-1$. Therefore, the length of the cycle (by counting $J_1$ twice) will be always an odd number. On the contrary, if cycle starts with chain $J_1\rightarrow J_{k+1}$, then to close it at $J_1$ the smallest required even length is $2k$.
				\item Consider primitive cycles that don't contain $J_{1}$. Obviously, such a cycle doesn't contain $J_{2},J_{3},\dots,J_{k-3}$ or $J_{k+4},J_{k+5},\dots, J_{2k-1}, J_{2k}$ since these vertices have red edges connecting them all the way to $J_{1}$. Additionally, this cycle cannot contain $J_{k+1}$ or $J_{k}$ since $J_{1}$ is the only vertex (besides $J_{k+1}$ itself) with a directed edge to $J_{k+1}$, and $J_{k+1}$ is the only vertex with directed edge to $J_{k}$. This leaves $4$ vertices: $J_{k-2}, J_{k-1}, J_{k+2}, J_{k+3}$. Since  $J_{k+2}{\color{Red}\rightarrow}J_{k-1}$, $J_{k+3}{\color{Red}\rightarrow}J_{k-2}$ and $J_{k+2}\leftrightharpoons J_{k-1}$, $J_{k+3}\leftrightharpoons J_{k-2}$, any cycle formed by these four vertices will consist of a starting vertex followed (or ending vertex preceded) by pairs $(J_{k+2}, J_{k-1})$, $(J_{k+3}, J_{k-2})$ added arbitrarily many times, and hence no cycles of even length can be produced. 
			\end{enumerate}	
			
		\item Case $(1)_{2}$: $f(1)=k \Rightarrow f(k+2)=k+1$, then we have the period $4$-suborbit\\ $\left\{ k-1, k+1, k+2, k+3 \right\}$, a contradiction.
		\item Case $(2)$: $f(k-1)=k+4 \Rightarrow f(k)=k+2$
		\item Case $(2)_{1}$: $f(1)=k+1 \Rightarrow f(k+2)=k$ then we have the period $2$-suborbit $\left\{ k, k+2 \right\}$, a contradiction.
		\item Case $(2)_{2}$: $f(1)=k \Rightarrow f(k+2)=k+1$; This produces a simple positive type $(2k+1)$-orbit given in \eqref{eq:len2kcase22cyc} and Fig.~\ref{fig:len2kcase22} with topological structure max-min-max. We repeat the argument from Case $(1)_{1}$. First we analyze the digraph to show that there are no primitive cycles of even length $\leq 2k-2$. 
We split the analysis into two cases:
		
			\begin{enumerate}
				\item Consider primitive cycles that contain $J_{1}$. Without loss of generality choose $J_1$ as a starting vertex. First assume that cycle starts with the edge $J_1\rightarrow J_{j}$, with $j$ taking any value between $k+3$ and $2k$. Since  $J_{2k+1-i}{\color{Red}\rightarrow}J_{i},\ i=1,...,k-2,$ any such cycle can be formed only by adding to starting vertex $J_1$ pairs $( J_{2k+1-i}, J_{i}), i=1,...,k-2$. Therefore, the length of the cycle (by counting $J_1$ twice) will be always an odd number. If the cycle starts with the edge $J_1\rightarrow J_{k+2}$, then the only difference from the previous case will be the addition of the pairs  $( J_{k+2}, J_{k-1})$ and/or $( J_{k+2}, J_{k})$ arbitrarily many times. Hence, only cycles of odd length will be produced. On the contrary, if the cycle starts with the chain $J_1\rightarrow J_{k+1}$ or $J_1\rightarrow J_{k}$, then to close it at $J_1$ the smallest required even length is $2k$.
				\item Consider primitive cycle that doesn't contain $J_{1}$. Obviously, such a cycle doesn't contain $J_{2},J_{3},\dots,J_{k-2}$ or $J_{k+3},J_{k+4},\dots, J_{2k-1}, J_{2k}$ since these vertices have red edges connecting them all the way to $J_{1}$. Additionally, this cycle cannot contain $J_{k+1}$ since $J_{1}$ is the only vertex (besides $J_{k+1}$ itself) with directed edge to $J_{k+1}$. This leaves $3$ vertices: $J_{k}, J_{k-1}, J_{k+2}$ connected as $J_{k} \leftrightharpoons J_{k+2} \leftrightharpoons J_{k-1}$. Therefore, this triple can only produce cycles when pairs $( J_{k+2}, J_{k-1})$ and $( J_{k+2}, J_{k})$ are added to a starting vertex. Therefore, no cycle of even length can be produced. 
			\end{enumerate}		
	\end{enumerate}
\end{proof}	
		
		Now, observe that when the length, $m$, of path \eqref{eq:secondminfc_index} is $2k-2$ it is comprised of $2k-2$ distinct intervals and thus there are $2$ additional intervals required to complete the periodic orbit of period $2k+1$. From path \eqref{eq:secondminfc_index} and the rules \eqref{eq:2minrulesForward} it follows that the relative distribution of the $2k-2$ intervals is in one of the following $2$ forms illustrated in Fig.~\ref{fig:case2km1Right} or Fig.~\ref{fig:case2km1Left}. Call the two missing intervals $\tilde{J}$ and $\hat{J}$. There are $2k-1$ slots in Fig.~\ref{fig:case2km1Right} or Fig.~\ref{fig:case2km1Left} where we can place \textit{each} of these extra intervals for a total of $(2k-1)^{2}$ pairs. However, since swapping the locations of $\tilde{J}$ and $\hat{J}$ does not affect the analysis let us consider the distribution given in Fig.~\ref{fig:case2km1Left} in the upper triangular matrix \eqref{eq:possSettings2km1} where $(i,j)$ indicates placing $\tilde{J}$ in position $i$ and $\hat{J}$ in position $j$.
		
\begin{figure}[htb]
	\centering
	\begin{tikzpicture}
		\draw (1,0)--(11.0,0);
		\foreach \num/\idx in {1,2,...,11}
			{
			\draw (\num,0.2)--(\num,-0.2); 
			}
		\foreach \label/\loc in {J_{r_{2k-3}}/1.5, \cdots/2.5, J_{r_{5}}/3.5, J_{r_{3}}/4.5, J_{r_{1}}/5.5, J_{r_{2}}/6.5, J_{r_{4}}/7.5, J_{r_{6}}/8.5, \cdots/9.5, J_{r_{2k-2}}/10.5}
			{
			\node[above] at (\loc, 0) {$\label$};
			}
			\node (5) at (5,0) {};
			\node (6) at (6,0) {};
			\node (7) at (7,0) {};
			\path[->] (6) edge[bend left=90] node{} (5);
			\path[->] (5) edge[bend left=65] node{} (7);			
	\end{tikzpicture}
	\caption{Relative positions of intervals in the sub-graph of length $m=2k-2$ when $J_{r_{2}}$ to the right of $J_{r_{1}}$}
	\label{fig:case2km1Right}
\end{figure}
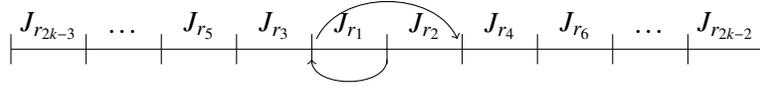

\noindent or

\begin{figure}[htb]
	\centering
	\begin{tikzpicture}
		\draw (1,0)--(11.0,0);
		\foreach \num/\idx in {1,2,...,11}
			{
			\draw (\num,0.2)--(\num,-0.2); 
			}
		\foreach \label/\loc in {J_{r_{2k-2}}/1.5, \cdots/2.5, J_{r_{6}}/3.5, J_{r_{4}}/4.5, J_{r_{2}}/5.5, J_{r_{1}}/6.5, J_{r_{3}}/7.5, J_{r_{5}}/8.5, \cdots/9.5, J_{r_{2k-3}}/10.5}
			{
			\node[above] at (\loc, 0) {$\label$};
			}
			\node (5) at (5,0) {};
			\node (6) at (6,0) {};
			\node (7) at (7,0) {};			
			\path[->] (6) edge[bend left=90] node{} (7);
			\path[->] (7) edge[bend left=65] node{} (5);				
	\end{tikzpicture}
	\caption{Relative positions of intervals in the sub-graph of length $m=2k-2$ when $J_{r_{2}}$ to the left of $J_{r_{1}}$}
	\label{fig:case2km1Left}
\end{figure}		
		
\begin{equation}
	\begin{matrix}
		(1,1) & (1,2) & (1,3) & \cdots & (1, 2k-3) & (1,2k-2) & (1,2k-1) \\
					& (2,2) & (2,3) & \cdots & (2, 2k-3) & (2,2k-2) & (2,2k-1) \\
					&       & (3,3) & \cdots & \vdots    &          & \vdots 	 \\			
					&	      &       & \ddots & (i, i)    & \cdots   & (i,2k-i+2) \\						
					&	      &       &        & \ddots	   & \vdots   & \vdots   \\						
					&       &       &        &           & \ddots   & (2k-1,2k-1) 
	\end{matrix}
	\label{eq:possSettings2km1}
\end{equation}		

The next lemma specifies all the entries $(i,j)$ in matrix \eqref{eq:possSettings2km1} , such that insertion of $(\tilde{J}, \hat{J})$ in $(i,j)$ can produce second minimal odd orbits.

\begin{lem}
	Fix the entry point, $i$, for $\tilde{J}$, then to produce second minimal $2k+1$ orbit, $\hat{J}$ can only be placed in
	\begin{enumerate}
		\item position $2k-1$ when $i=1$,
		\item positions $2k-i$ or $2k-i+1$ for $1 < i < k$, 
		\item and position $k+1$ when $i=k$.
	\end{enumerate}
	\label{lem:JhatJtilde}
\end{lem}	
\begin{proof}
	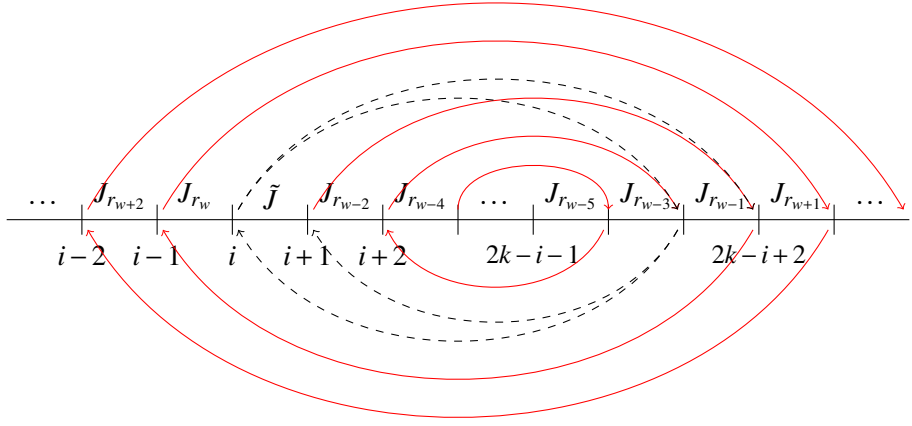
\begin{figure}[ht]%
		\centering
			\begin{tikzpicture}
				\draw (1,0)--(13.0,0);
				\foreach \num/\idx in {2,3,...,12}
					{
					\draw (\num,0.2)--(\num,-0.2); 
					\node (\num) at (\num,0) {};
					}
					\node (1) at (1,0) {};
					\node (13) at (13,0) {};
					\node[below] at (2,-0.21) {$i-2$};
					\node[below] at (3,-0.21) {$i-1$};
					\node[below] at (4,-0.21) {$i$};
					\node[below] at (5,-0.21) {$i+1$};
					\node[below] at (6,-0.21) {$i+2$};
					
					\node[below] at (8,-0.21) {\small{$2k-i-1$}};
					\node[below] at (11,-0.21) {\small{$2k-i+2$}};						
					
				\foreach \label/\loc in {\cdots/1.5, J_{r_{w+2}}/2.5, J_{r_{w}}/3.5, \tilde{J}/4.5, J_{r_{w-2}}/5.5, J_{r_{w-4}}/6.5, \cdots/7.5, J_{r_{w-5}}/8.5, J_{r_{w-3}}/9.5, J_{r_{w-1}}/10.5, J_{r_{w+1}}/11.5, \cdots/12.5}
					{
					\node[above] at (\loc, 0) {$\label$};
					}
				
				\path[->] (7) edge[red, bend left=90] node{} (9);
				\path[->] (9) edge[red, bend left=65] node{} (6);
				\path[->] (6) edge[red, bend left=60] node{} (10);		
				\path[->] (5) edge[red, bend left=60] node{} (11);		
				\path[->] (3) edge[red, bend left=60] node{} (12);	
				\path[->] (11) edge[red, bend left=60] node{} (3);

				\path[dashed,->] (10) edge[bend left=60] node{} (4);
				\path[dashed,->] (10) edge[bend left=60] node{} (5);
				\path[dashed,->] (4) edge[bend left=60] node{} (10);
				\path[dashed,->] (4) edge[bend left=60] node{} (11);
				\path[->] (12) edge[red, bend left=60] node{} (2);
				\path[->] (2) edge[red, bend left=60] node{} (13);
			\end{tikzpicture}	
			\caption{Relative ordering when $m=2k-2$ and $\tilde{J}$ is in position $i$}%
			\label{fig:fixJtildein3}%
	\end{figure}				

Let $1 < i < k$ for $k>3$. Let the interval immediately to the left of $\tilde{J}$ be called $J_{r_{w}}$, or, the $w$-th distinct interval in path \eqref{eq:secondminfc_index}. From the relative positions of the intervals in Fig.~\ref{fig:secminint2} it is clear that $w = 2(k-i+1)$. As depicted in Fig.~\ref{fig:fixJtildein3}, the intervals to the right of $\tilde{J}$, have a set path. $J_{r_{w-4}}$ maps only to $J_{r_{w-3}}$, $J_{r_{w-2}}$ maps only to $J_{r_{w-1}}$, etc. Note that while the exact points that these intervals map to might change upon the insertion of $\hat{J}$, the overall structure must remain the same in order to preserve path \eqref{eq:secondminfc_index}.
	
	However, this pattern can no longer be continued indefinitely for interval $J_{r_{w}}$. $J_{r_{w}}$ must be mapped to $J_{r_{w+1}}$, by definition. This implies that one end point of $J_{r_{w}}$ is mapped arbitrarily to the left of $J_{r_{w+1}}$, and the other endpoint is mapped arbitrarily to the right of $J_{r_{w+1}}$. According to \eqref{eq:2minrulesForward}, $J_{r_{w}}$ can't map to any odd interval greater than $J_{r_{w+1}}$. Assuming that $J_{r_{w+1}}$ isn't the rightmost interval, or in other words $w+1<2k-3$, then the endpoint of $J_{r_{w}}$ that maps to the right of $J_{r_{w+1}}$, must map to a point separating $J_{r_{w+1}}$ and $J_{r_{w+3}}$. This could either be a point directly in between $J_{r_{w+1}}$ and $J_{r_{w+3}}$, or a new point between $J_{r_{w+1}}$ and $J_{r_{w+3}}$, upon the insertion of $\hat{J}$.
	
	Note that if $w+1=2k-3$, or in other words $J_{r_{w+1}}$ is the rightmost interval, then $J_{r_{w+3}}$ no longer exists. Rather than mapping to a point separating $J_{r_{w+1}}$ and $J_{r_{w+3}}$, an endpoint of $J_{r_{w}}$ will just map to the right of $J_{r_{w+1}}$.
	
	While it is clear how one endpoint of $J_{r_{w}}$ will map to the right of $J_{r_{w+1}}$, it is much less clear how an endpoint of $J_{r_{w}}$ will map to the left of $J_{r_{w+1}}$. According to \eqref{eq:2minrulesBackward}, $J_{r_{w}}$ can not map to $J_{r_{1}}$ or any $J_{r_{k}}$ where $k<w$ and is even. Thus, the arbitrary point to the left of $J_{r_{w+1}}$ that an endpoint of $J_{r_{w}}$ must map to, must be separating $J_{r_{1}}$ and $J_{r_{w+1}}$. Thus, the missing interval $\hat{J}$ must be inserted between $J_{r_{1}}$ and $J_{r_{w+1}}$. Note that according to the previously mentioned positional notation, this includes all the positions $k+1$, $k+2$, \ldots, $2k-i$, $2k-i+1$.
	
	Suppose that the interval $\hat{J}$ is inserted into a position to the left of $2k-i$. Assume, without loss of generality, that the interval $\hat{J}$ is inserted into position $2k-i-1$ $($between intervals $J_{r_{w-3}}$ and $J_{r_{w-5}}$ if $w \geq 6$, and between $J_{r_{1}}$ and $J_{r_{2}}$ if $w=4$), as the contradiction will be the same. The interval immediately to the right of $\hat{J}$ has some trouble mapping to the interval indexed one greater than itself. In the situation where $\hat{J}$ is placed in position $2k-i-1$, the interval immediately to the right of $\hat{J}$ is $J_{r_{w-3}}$, which has trouble mapping to $J_{r_{w-2}}$. One endpoint of $J_{r_{w-3}}$ must map to the left of $J_{r_{w-2}}$, while the other must map to the right of $J_{r_{w-2}}$. Again, according to \eqref{eq:2minrulesBackward}, $J_{r_{w-3}}$ can't map to a lesser odd interval, or $J_{r_{1}}$. Thus, the only available point to the left of $J_{r_{1}}$ and to the right of $J_{r_{w-2}}$, is the point indexed $k+1$, or the left endpoint of $J_{r_{1}}$. However, it is important to note that the largest indexed interval, $J_{r_{2k-2}}$, must also map back to $J_{r_{1}}$. And according to \eqref{eq:2minrulesBackward}, $J_{r_{2k-2}}$ can't map to a lesser even interval, specifically including $J_{r_{2}}$. Therefore, an endpoint of $J_{r_{2k-2}}$ must map to the left of $J_{r_{1}}$, but can not map to the left of $J_{r_{2}}$. The only point that fits this description is the one indexed $k+1$. Thus, the point $k+1$ is already taken, and an endpoint of $J_{r_{w-3}}$ can not map to the point indexed $k+1$. Furthermore, there are no open points that are both to the right of $J_{r_{w-2}}$ and to the left of $J_{r_{1}}$. This is an immediate contradiction, as it is now impossible for an endpoint of $J_{r_{w-3}}$ to map to the right of $J_{r_{w-2}}$, making it impossible for $J_{r_{w-3}}$ to contain $J_{r_{w-2}}$. 
	
	Now suppose instead $w=4$. If this is the case, then $\hat{J}$ is inserted between $J_{r_{1}}$ and $J_{r_{2}}$. Again, the image of $J_{r_{1}}$ must contain only $J_{r_{1}}$ and $J_{r_{2}}$. Thus, one endpoint of $J_{r_{1}}$ must map to the right of $J_{r_{1}}$, but can not map to the right of $J_{r_{3}}$. Therefore, the left endpoint of $J_{r_{1}}$ must map to the point separating $J_{r_{1}}$ and $J_{r_{3}}$. Now, $J_{r_{2}}$ must map to $J_{r_{3}}$, but can not map back to $J_{r_{1}}$. Thus, one of the endpoints of $J_{r_{2}}$ must map to a point separating $J_{r_{1}}$ and $J_{r_{3}}$. Both the left endpoint of of $J_{r_{1}}$ and an endpoint of $J_{r_{2}}$ must map to points separating $J_{r_{1}}$ and $J_{r_{3}}$. There is only one point separating $J_{r_{1}}$ and $J_{r_{3}}$, so this is an immediate contradiction.
	
	Therefore, it is impossible for $\hat{J}$ to be inserted to the left of position $2k-i$, because the interval immediately to the right of $\hat{J}$ can no longer map to the interval indexed one greater than itself, in the case of $w \geq 6$. In the case of $w=4$, a separate contradiction arises when both the left endpoint of of $J_{r_{1}}$ and an endpoint of $J_{r_{2}}$ must map to the singular point separating $J_{r_{1}}$ and $J_{r_{3}}$. The only two possible positions of $\hat{J}$ that can produce valid second minimal orbits when $\tilde{J}$ is inserted into position $i$, are the positions $2k-i$ and $2k-i+1$.
		
		

		Suppose that $\tilde{J}$ is inserted into position $k$. By \eqref{eq:2minrulesForward}, the image of $J_{r_{1}}$ must contain itself. This means, by definition that one endpoint of $J_{r_{1}}$ must map to the right of itself, and the other endpoint of $J_{r_{1}}$ must map to the left of itself. However, by \eqref{eq:2minrulesForward}, $J_{r_{1}}$ can not map to $J_{r_{3}}$. Thus, the endpoint of $J_{r_{1}}$ that maps to the right of itself, must map to the left of $J_{r_{3}}$. In other words, one endpoint of $J_{r_{1}}$ must map to a point separating $J_{r_{1}}$ and $J_{r_{3}}$. By \eqref{eq:2minrulesForward}, the image of $J_{r_{2}}$ must contain $J_{r_{3}}$. Again, this means that one endpoint of $J_{r_{2}}$ maps to the left of $J_{r_{3}}$, and the other endpoint of $J_{r_{2}}$ maps to the right of $J_{r_{3}}$. However, by \eqref{eq:2minrulesBackward}, the image of $J_{r_{2}}$ can not contain $J_{r_{1}}$. Thus, the endpoint of $J_{r_{2}}$ that maps to the left of $J_{r_{3}}$ can not map to the left of $J_{r_{1}}$. In other words, an endpoint of $J_{r_{2}}$ must map to points separating $J_{r_{1}}$ and $J_{r_{3}}$. Both an endpoint of $J_{r_{1}}$ and an endpoint $J_{r_{2}}$ must map to points separating $J_{r_{1}}$ and $J_{r_{3}}$. Note that, when $\tilde{J}$ was inserted into position $k$, $J_{r_{1}}$ and $J_{r_{2}}$ no longer shared an endpoint. Thus, the two endpoints that map to points separating $J_{r_{1}}$ and $J_{r_{3}}$ must necessarily be two different points. Thus, there must be at least two different points separating $J_{r_{1}}$ and $J_{r_{3}}$. This can only be achieved by inserting $\hat{J}$ between $J_{r_{1}}$ and $J_{r_{3}}$, which is position $k+1$. Or in other words, if $\tilde{J}$ is inserted into position $k$, a valid second minimal orbit can only be constructed if $\hat{J}$ is inserted into position $k+1$.
		
		Finally, suppose that $\tilde{J}$ is inserted into position $1$. There are two options here:
		\begin{enumerate}
			\item $\hat{J}$ is inserted arbitrarily to the left of $J_{r_{1}}$.
			\item $\hat{J}$ is inserted arbitrarily to the right of $J_{r_{1}}$.
		\end{enumerate}
		
		Suppose that $\hat{J}$ is inserted arbitrarily to the left of $J_{r_{1}}$. If $\hat{J}$ is in any position $j$, where $2 \leq j \leq k$. By simply changing the notation, where $\hat{J}$ becomes $\tilde{J}$ and $\tilde{J}$ becomes $\hat{J}$, we suddenly have the cases already addressed earlier in the proof, with $\tilde{J}$ being arbitrarily between $J_{r_{2k-2}}$ and $J_{r_{1}}$. Note that for none of the cases where $2 \leq i \leq k$, position $1$ was not a valid position for $\hat{J}$. Thus, the only case that hasn't been considered is when both $\tilde{J}$ and $\hat{J}$ are in position $1$. For the sake of simplicity, assume that $\hat{J}$ is the first interval, and $\tilde{J}$ is the interval between $\hat{J}$ and $J_{r_{2k-2}}$.
		
		It follows from \eqref{eq:2minrulesForward} that the image of $J_{r_{1}}$ must contain only itself and $J_{r_{2}}$. This is only possible if the left endpoint of $J_{r_{1}}$, indexed $k+2$ maps to the right endpoint of $J_{r_{1}}$, which is indexed $k+3$; and the right endpoint of $J_{r_{1}}$ maps to the left endpoint of $J_{r_{2}}$, indexed $k+1$. By \eqref{eq:2minrulesForward}, $J_{r_{2}}$ must contain $J_{r_{3}}$ and no odd interval with a greater index, which is only possible if the left endpoint of $J_{r_{2}}$, indexed $k$, maps to the right endpoint of $J_{r_{3}}$, indexed $k+4$. Furthermore, by \eqref{eq:2minrulesForward} every even interval can contain only the odd interval indexed one greater than itself and no greater odd interval, and conversely, every odd interval can contain only the even interval indexed one greater than itself and no greater even interval. This fact forces the intervals to follow the structure depicted in \cite{Stefan1977}. This pattern continues until the left endpoint of $J_{r_{2k-3}}$ maps to the right endpoint of $J_{r_{2k-2}}$, and the left endpoint of $J_{r_{2k-4}}$ maps to the right of $J_{r_{2k-3}}$. This construction yields Fig.~\ref{fig:fixJtildein1}, where the solid red lines represent the Stefan-like structure present when $\tilde{J}$ and $\hat{J}$ are both placed in Position $1$.
		
	\begin{figure}[ht]%
		\centering
			\begin{tikzpicture}
				\draw (1,0)--(13.0,0);
				\foreach \num/\idx in {1,2,...,13}
					{
					\draw (\num,0.2)--(\num,-0.2); 
					\node (\num) at (\num,0) {};
					}
					\node (1) at (1,0) {};
					\node (13) at (13,0) {};
					\node[below] at (1,-0.21) {\small{$1$}};
					\node[below] at (2,-0.21) {\small{$2$}};
					\node[below] at (3,-0.21) {\small{$3$}};
					\node[below] at (4,-0.21) {\small{$4$}};
					\node[below] at (5,-0.21) {\small{$5$}};
					\node[below] at (6,-0.21) {\small{$k$}};
					\node[below] at (7,-0.21) {\small{$k+1$}};
					\node[below] at (8,-0.21) {\small{$k+2$}};
					\node[below] at (9,-0.21) {\small{$k+3$}};
					\node[below] at (10,-0.21) {\small{$k+4$}};
					\node[below] at (11,-0.21) {\small{$2k-1$}};
					\node[below] at (12,-0.21) {\small{$2k$}};
					\node[below] at (13,-0.21) {\small{$2k+1$}};		
					
				\foreach \label/\loc in {\hat{J}/1.5, \tilde{J}/2.5, J_{r_{2k-2}}/3.5, J_{r_{2k-4}}/4.5, \cdots/5.5, J_{r_{4}}/6.5, J_{r_{2}}/7.5, J_{r_{1}}/8.5, J_{r_{3}}/9.5, \cdots/10.5, J_{r_{2k-5}}/11.5, J_{r_{2k-3}}/12.5}
					{
					\node[above] at (\loc, 0) {$\label$};
					}
				
				\path[->] (8) edge[red, bend left=90] node{} (9);
				\path[->] (9) edge[red, bend left=90] node{} (7);
				\path[->] (7) edge[red, bend left=90] node{} (10);
				\path[->] (10) edge[red, bend left=90] node{} (6);
				
				\path[->] (11) edge[red, bend left=90] node{} (5);
				\path[->] (12) edge[red, bend left=90] node{} (4);
				\path[->] (5) edge[red, bend left=90] node{} (12);
				\path[->] (4) edge[red, bend left=90] node{} (13);
				\path[dashed,->] (3) edge[red, bend left=90] node{} (8);

			\end{tikzpicture}	
			\caption{Fixed ordering when for intervals $J_{r_{1}}$ through $J_{r_{2k-4}}$, when both $\tilde{J}$ and $\hat{J}$ are in position $1$.}
			\label{fig:fixJtildein1}%
	\end{figure}
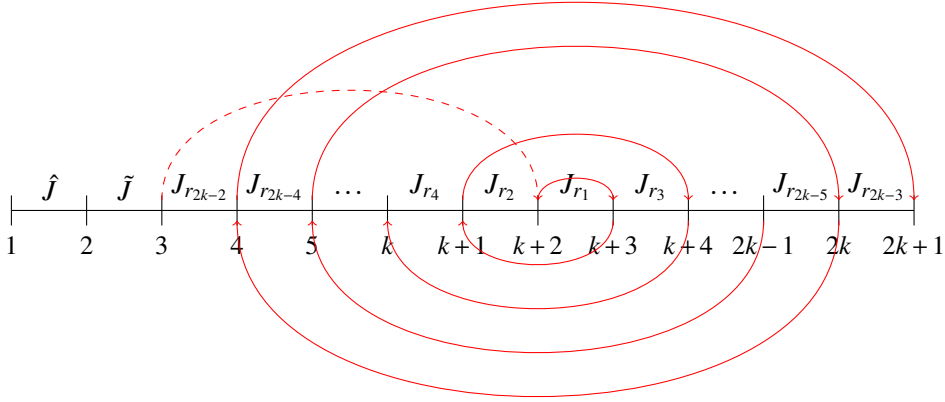
		
		By \eqref{eq:2minrulesForward}, the image of $J_{r_{2k-2}}$ must contain $J_{r_{1}}$, but can not map to a lesser even interval, specifically $J_{r_{2}}$. Therefore, the left endpoint of $J_{r_{2k-2}}$ must map somewhere to the left of $J_{r_{1}}$, but can not map to the left of $J_{r_{2}}$. In other words, it must map to a point separating $J_{r_{1}}$ and $J_{r_{2}}$. The only point that meets this condition is indexed $k+2$. Therefore, the left endpoint of $J_{r_{2k-2}}$ must map to the point indexed $k+2$, shown by the dotted red line in Fig.~\ref{fig:fixJtildein1}.
		
		By \eqref{eq:2minrulesForward}, the image of $J_{r_{2k-3}}$ must contain $J_{r_{2k-2}}$. Thus, the right endpoint of $J_{r_{2k-3}}$ must map to the left of $J_{r_{2k-2}}$. However, note if the right endpoint of $J_{r_{2k-3}}$ maps immediately to the left of $J_{r_{2k-2}}$, to the point indexed $3$, then two closed sub-orbits of length $2$, and of the form $\left\{ 1, 2 \right\}$. Thus, the right endpoint of $J_{r_{2k-3}}$ can map either to the point indexed $1$ or the point indexed $2$. Investigating both of these cases individually leads to quick contradictions.
		
		Suppose that the right endpoint of $J_{r_{2k-3}}$ maps to the left endpoint of $\hat{J}$, or the point indexed $1$. The only two 'open' points, or points that don't already have a point mapping to them, are the ones indexed $2$ and $3$. Because point $2$ can not map to itself, it must map to the only other open point, indexed $3$. The final point, indexed $1$, can now only map to point $2$, completing the cyclic permutation. However, this can instantly be shown to be an invalid cyclic permutation, due to the presence of a $3$ orbit, of the form $\tilde{J} \rightarrow J_{r_{2k-2}} \rightarrow J_{r_{2k-3}} \rightarrow \tilde{J}$.
		
		Since the left endpoint of $J_{r_{2k-3}}$ can not map to the point indexed $1$, the case where the left endpoint of $J_{r_{2k-3}}$ maps to the point indexed $2$, is considered. Now, the only two open points are the ones indexed $1$ and $3$. Because point $1$ can not map to itself, it must map to the only other open point, indexed $3$. The final point, indexed $2$, can now only map to point $1$, completing the cyclic permutation. However, this can instantly be shown to be an invalid cyclic permutation, due to the presence of a $3$ orbit, of the form $\hat{J} \rightarrow \tilde{J} \rightarrow \hat{J} \rightarrow \hat{J}$. 
		
		This exhausts all possible options for points which the right endpoint of $J_{r_{2k-3}}$ can map to, proving that it is impossible to form a valid cyclic permutation, when $\tilde{J}$ and $\hat{J}$ are both in position $1$. Furthermore, all cases where $\tilde{J}$ is in position $1$ and $\hat{J}$ is in an arbitrary position $j$, where $2 \leq j \leq k$ have been proven to lead to contradictions. Thus any case where $\tilde{J}$ is in position $1$ and $\hat{J}$ is inserted arbitrarily to the left of $J_{r_{1}}$, can not lead to the construction of a valid second minimal odd orbit. 
		
		A valid second minimal odd orbit for the case where $\tilde{J}$ is in position $1$, can only be constructed when $\hat{J}$ is inserted arbitrarily to the right of $J_{r_{1}}$. Consider the case when $\hat{J}$ is inserted between $J_{r_{1}}$ and $J_{r_{2k-3}}$. Now assume without loss of generality that $\hat{J}$ is inserted to the position immediately to the left of $J_{r_{2k-3}}$, as the contradiction will be the same. The interval immediately to the right of $\hat{J}$, in this case $J_{r_{2k-3}}$, has trouble mapping to the interval indexed one greater than itself, in this case $J_{r_{2k-2}}$. While it is clear how an endpoint of $J_{r_{2k-3}}$ will map to the left of $J_{r_{2k-2}}$, it is much less clear how an endpoint will map to the right of $J_{r_{2k-3}}$. Again, according to \eqref{eq:2minrulesBackward}, the image of $J_{r_{2k-3}}$ can't contain to $J_{r_{1}}$. Thus, an endpoint of $J_{r_{2k-3}}$ must map to a point separating $J_{r_{2k-2}}$ and $J_{r_{1}}$. The only point that fits this description is the one indexed $k+1$. However, it is important to note that the largest indexed interval, $J_{r_{2k-2}}$, must also map back to $J_{r_{1}}$. And according to \eqref{eq:2minrulesBackward}, $J_{r_{2k-2}}$ can't map to a lesser even interval, specifically including $J_{r_{2}}$. Therefore, an endpoint of $J_{r_{2k-2}}$ must map to the left of $J_{r_{1}}$, but can not map to the left of $J_{r_{2}}$. The only point that fits this description is the one indexed $k+1$, which is already taken. This is an immediate contradiction, as both an endpoint of $J_{r_{2k-3}}$ and an endpoint of $J_{r_{2k-2}}$ must map to the point indexed $k+1$. Therefore, it is impossible for $\hat{J}$ to be placed between $J_{r_{1}}$ and $J_{r_{2k-3}}$. The only possible way to make the construction valid second minimal odd orbits possible when $\tilde{J}$ is placed in position one, is by placing $\hat{J}$ into the position $2k-1$, or immediately to the right of $J_{r_{2k-3}}$.		
		
To complete the proof we show there are no valid settings $(i,j)$ for $k < i \leq 2k-1$, $j \geq i$. Note that, \eqref{eq:possSettings2km1} only includes $(i,j)$ pairs where $j\geq k$, if $i>k$. Thus, both $i$ and $j$ are to the right of $J_{r_{1}}$. Assume for the sake of simplicity, that the interval closer to $J_{r_{1}}$ is the one labeled $\tilde{J}$. All intervals between $J_{r_{1}}$ and $\tilde{J}$ will map according to the minimal structure \ref{thm:stefan} \cite{block92}, as described earlier. The Stefan structure comes to an end when the interval immediately to the left of $\tilde{J}$, $J_{r_{2i-2k-1}}$, maps to $J_{r_{2i-2k}}$. Now, the interval to the right of $\tilde{J}$, $J_{r_{2i-2k+1}}$, has trouble mapping to the interval $J_{r_{2i-2k+2}}$. $J_{r_{2i-2k+1}}$ can not map to $J_{r_{1}}$, according to the rules \eqref{eq:2minrulesForward} and \eqref{eq:2minrulesBackward}, so one of its endpoints must map to a point separating $J_{r_{1}}$ and $J_{r_{2i-2k+2}}$. The only such point is $k$. However, as discussed before, $k$ must be mapped to by an endpoint of the interval $J_{r_{2k-2}}$. Thus there are no open points between $J_{r_{1}}$ and $J_{r_{2i-2k+2}}$. Note that, because $\hat{J}$ is inserted to the right of $\tilde{J}$, it is impossible for an endpoint of $J_{r_{2i-2k+1}}$ to map to the right of $J_{r_{2i-2k+2}}$. Thus it is impossible for $J_{r_{2i-2k+1}}$ to contain $J_{r_{2i-2k+2}}$, when both $\tilde{J}$ and $\hat{J}$ are to the right of $J_{r_{1}}$, which is a clear contradiction.		

%
%
%

		\begin{lem}
		Placing $\tilde{J}$ and $\hat{J}$ in setting $(1,2k-1)$ produces exactly $2$ second minimal orbits. These are given in listings \eqref{eq:case1-2km1-1} and \eqref{eq:case1-2km1-2} and the corresponding digraphs are presented in Fig.~\ref{fig:case1-2km1-1} and Fig.~\ref{fig:case1-2km1-2} respectively.
		
		\small{
		\begin{equation}
			\left(\begin{array}{cccccccccc}   
				1    & 2   & 3  & \cdots & k+1 & k+2 & \cdots & 2k-1 & 2k & 2k+1 \\ 
				2k+1 & k+1 & 2k & \cdots & k+2 & k 	 & \cdots & 3    & 1  & 2 \\
			\end{array} \right)	
			\label{eq:case1-2km1-1}	
		\end{equation}}

		\small{
		\begin{equation}
			\left(\begin{array}{cccccccccc}   
				1  & 2   & 3    & 4 	 & \cdots & k+1 & k+2 & \cdots & 2k & 2k+1 \\ 
				2k & k+1 & 2k+1 & 2k-1 & \cdots & k+2 & k 	 & \cdots & 2  & 1 \\
			\end{array} \right)	
			\label{eq:case1-2km1-2}
		\end{equation}}				
	
		\input{digraphs/digraph1-2km1-1.tex}
		
		\input{digraphs/digraph1-2km1-2.tex}
		
		\label{lem:case1-2km1}
	\end{lem}
	\begin{proof}

		Having inserted $(\tilde{J}, \hat{J})$ in position $(1,2k-1)$ we have the full interval distribution given in Fig.~\ref{fig:set1c2km1}. Then, using the path \eqref{eq:secondminfc_index} and the rules in \eqref{eq:2minrulesForward}, and \eqref{eq:2minrulesBackward} we observe that the images of the elements of the cycle from $4$ to $2k-1$ are uniquely defined following Stefan structure as it is demonstrated in \eqref{eq:case1-2km1-1-4-2k-1}.	 
		\small{
		\begin{equation}
			\left(\begin{array}{cccccccccc}   
				4  & \cdots & k & k+1 & k+2 & \cdots & 2k-1 \\ 
				2k-1 & \cdots & k+3 & k+2 & k 	 & \cdots & 3  \\
			\end{array} \right)	
			\label{eq:case1-2km1-1-4-2k-1}	
		\end{equation}}
The alteration appears only in images of the elements $1,3,2k,2k+1$. By using the path \eqref{eq:secondminfc_index} and the rules in \eqref{eq:2minrulesForward}, and \eqref{eq:2minrulesBackward} again, we construct the potential cyclic permutation \eqref{eq:cycperm2km1} and analyze which of the available choices lead to valid second minimal odd orbits.
		
	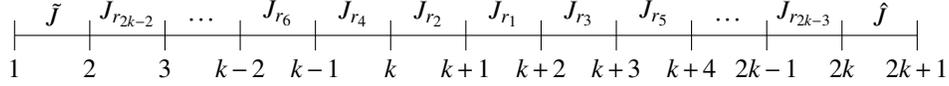
\begin{figure}[htb]
		\begin{tikzpicture}
			\draw (1,0)--(13.0,0);
			\foreach \num/\idx in {1,2,...,13}
				{
				\draw (\num,0.2)--(\num,-0.2); 
				}
			\foreach \label/\loc in {\tilde{J}/1.5, J_{r_{2k-2}}/2.5, \cdots/3.5, J_{r_{6}}/4.5, J_{r_{4}}/5.5, J_{r_{2}}/6.5, J_{r_{1}}/7.5, J_{r_{3}}/8.5, J_{r_{5}}/9.5, \cdots/10.5, J_{r_{2k-3}}/11.5, \hat{J}/12.5}
				{
				\node[above] at (\loc, 0) {$\label$};
				}
			\foreach \label/\loc in {1/1, 2/2, 3/3, k-2/4, k-1/5, k/6, k+1/7, k+2/8, k+3/9, k+4/10, 2k-1/11, 2k/12, 2k+1/13 }
				{
				\node[below] at (\loc, -0.2) {\small{$\label$}};
				}		
		\end{tikzpicture}
		\caption{Complete interval for case when length is $2k-1$ with $\left | B^{-} \right | > \left | B^{+} \right |$ and missing intervals are at setting $(1,2k-1)$}
		\label{fig:set1c2km1}
	\end{figure}

	\small{
	\begin{equation}
		\left(\begin{array}{ccccccccc}   
			1 & 2 & 3 & \cdots & k+1 & \cdots & 2k-1 & 2k & 2k+1 \\ 
			\begin{matrix} 2k+1\\ 2k \\ 2 \end{matrix} & k+1 & \begin{matrix} 2k\\ 2k+1 \end{matrix} & \cdots & k+2 & \cdots & 3  & \begin{matrix} 1\\ 2 \end{matrix} & \begin{matrix} 2\\ 1 \\ 2k \end{matrix} \\
		\end{array} \right)
		\label{eq:cycperm2km1}
	\end{equation}}

	\begin{enumerate}
		\item Case $(1)$: Choosing $f(2k+1)=2 \Rightarrow f(2k)=1, f(1)=2k+1, f(3)=2k$; This leads to a valid second minimal orbit with the topological structure min-max-min, the associated digraph is presented in Fig.~\ref{fig:case1-2km1-1} and the cyclic permutation is listed in \eqref{eq:case1-2km1-1}. Next we analyze the digraph to show that there are no primitive cycles of even length $\leq 2k-2$, which would imply by Straffin's lemma an existence of odd periodic orbit of length $\leq 2k-3$. From Lemms~\ref{thm:convStraffin} it then follows that the the $P$-linearization of the orbit \eqref{eq:case1-2km1-1} presents an example of continuous map with second minimal $(2k+1)$-orbit.

		\begin{enumerate}
			\item Consider primitive cycles that contain $J_{1}$. Without loss of generality choose $J_1$ as the starting vertex. First assume that cycle doesn't contain $J_{k+1}$. Since  $J_{2k+1-i}{\color{Red}\rightarrow}J_{i},\ i=1,...,k-1; J_{2k-1} \rightarrow J_1$ any such cycle can be formed only by adding to starting vertex $J_1$ pairs $(J_{2k-1}, J_1)$, $( J_{2k+1-i}, J_{i}), i=1,...,k-1$. Therefore, the length of the cycle (by counting $J_1$ twice) will be always an odd number. On the contrary, if cycle contains $J_{k+1}$, then to close it at $J_1$ the smallest required even length is $2k+2$.
			
			\item Consider primitive cycles that doesn't contain $J_{1}$, but contain $J_2$. Without loss of generality choose $J_2$ as the starting vertex. First assume that cycle doesn't contain $J_{k+1}$. Since  $J_{2k+1-i}{\color{Red}\rightarrow}J_{i},\ i=2,...,k-1,$ any such cycle can be formed only by adding to starting vertex $J_2$ pairs $( J_{2k+1-i}, J_{i}), i=2,...,k-1$. Therefore, the length of the cycle (by counting $J_2$ twice) will be always an odd number. On the contrary, if cycle contains $J_{k+1}$, then to close it at $J_2$ the smallest required even length is $2k$. Finally, it is easy to see that by excluding $J_1$ and $J_2$ from the cycle, due to red edges all the vertices but $J_{k+1}$ must be also excluded, and cycle at $J_{k+1}$ is the only possibility. 

		\end{enumerate}

		\item Case $(2)$: $f(2k+1)=1 \Rightarrow f(2k)=1, f(1)=2k+1, f(3)=2k$; This leads to a valid second minimal orbit with the topological structure min-max, the associated digraph is presented in Fig.~\ref{fig:case1-2km1-2} and the cyclic permutation is listed in \eqref{eq:case1-2km1-2}. Next we prove as in previous case that there are no primitive cycles of even length $\leq 2k-2$, and therefore according to Lemms~\ref{thm:convStraffin} $P$-linearization of the orbit \eqref{eq:case1-2km1-1} presents an example of continuous map with second minimal $(2k+1)$-orbit.

		\begin{enumerate}
			\item Consider primitive cycles that contain $J_{1}$. Without loss of generality choose $J_1$ as the starting vertex. First assume that cycle doesn't contain $J_{k+1}$. Since  $J_{2k+1-i}{\color{Red}\rightarrow}J_{i},\ i=1,...,k-1,$ any such cycle can be formed only by adding to starting vertex $J_1$ pairs $( J_{2k+1-i}, J_{i}), i=1,...,k-1$. Therefore, the length of the cycle (by counting $J_1$ twice) will be always an odd number. On the contrary, if cycle contains $J_{k+1}$, then to close it at $J_1$ the smallest required even length is $2k+2$.
			
			\item Consider primitive cycles that doesn't contain $J_{1}$, but contain $J_2$. Without loss of generality choose $J_2$ as the starting vertex. First assume that cycle doesn't contain $J_{k+1}$. Since  $J_{2k+1-i}{\color{Red}\rightarrow}J_{i},\ i=2,...,k-1,$ any such cycle can be formed only by adding to starting vertex $J_2$ pairs $( J_{2k+1-i}, J_{i}), i=2,...,k-1$. Therefore, the length of the cycle (by counting $J_2$ twice) will be always an odd number. On the contrary, if cycle contains $J_{k+1}$, then to close it at $J_2$ the smallest required even length is $2k$. Finally, it is easy to see that by excluding $J_1$ and $J_2$ from the cycle, due to red edges all the vertices but $J_{k+1}$ must be also excluded, and cycle at $J_{k+1}$ is the only possibility. 

		\end{enumerate}

		\item Case $(3)$: $f(2k+1)=2k\Rightarrow f(3)=2k+1, f(1)=2, f(2k) = 1$. Produced cyclic permutation contains the subgraph $J_{2k}\rightarrow J_{1}\rightarrow J_{3}\rightarrow J_{2k}$. According to Straffin's lemma this subgraph implies the existence of period 3-orbit, which is a contardiction. 
	\end{enumerate}		
\end{proof}

\begin{lem}
	Placing $\tilde{J}$ and $\hat{J}$ in setting $(2,2k-1)$ produces exactly $3$ second minimal orbits. These are given in listings \eqref{eq:case1-2km1-1}, \eqref{eq:case2-2km1-2}, and \eqref{eq:case2-2km1-3} and the corresponding digraphs are presented in Fig.~\ref{fig:case1-2km1-1}, Fig.~\ref{fig:case2-2km1-2}, and Fig.~\ref{fig:case2-2km1-3} respectively.

	\small{
	\begin{equation}
		\left(\begin{array}{ccccccccccc}   
			1    & 2   & 3    & 4 	 & \cdots & k+1 & k+2 & \cdots & 2k-1 & 2k & 2k+1 \\ 
			k+1  & 2k  & 2k+1 & 2k-1 & \cdots & k+2 & k 	& \cdots & 3    & 1  & 2 \\
		\end{array} \right)	
		\label{eq:case2-2km1-2}	
	\end{equation}}		
	
	\small{
	\begin{equation}
		\left(\begin{array}{ccccccccccc}   
			1    & 2    & 3  & \cdots & k+1 & k+2 & \cdots & 2k-2 & 2k-1 & 2k & 2k+1 \\ 
			k+1  & 2k+1 & 2k & \cdots & k+2 & k   & \cdots & 4    & 2    & 1  & 3 \\
		\end{array} \right)	
		\label{eq:case2-2km1-3}	
	\end{equation}}		
	
	\input{digraphs/digraph2-2km1-2.tex}
	
	\input{digraphs/digraph2-2km1-3.tex}	
	
	\label{lem:case2-2km1}
\end{lem}
\begin{proof}
Having inserted $(\tilde{J}, \hat{J})$ in position $(2,2k-1)$ and by using the path \eqref{eq:secondminfc_index} and the rules in \eqref{eq:2minrulesForward}, and \eqref{eq:2minrulesBackward} we observe that the images of the elements of the cycle from $4$ to $2k-2$ are uniquely defined following Stefan structure as it is demonstrated in \eqref{eq:case1-2km1-2-4-2k-1}.	 
		\small{
		\begin{equation}
			\left(\begin{array}{cccccccccc}   
				4  & \cdots & k & k+1 & k+2 & \cdots & 2k-2 \\ 
				2k-1 & \cdots & k+3 & k+2 & k 	 & \cdots & 4  \\
			\end{array} \right)	
			\label{eq:case1-2km1-2-4-2k-1}	
		\end{equation}}
The alteration appears only in images of the elements $1,2,3,2k-1,2k+1$. By using the path \eqref{eq:secondminfc_index} and the rules in \eqref{eq:2minrulesForward}, and \eqref{eq:2minrulesBackward} again, we construct the potential cyclic permutation \eqref{eq:case2-2km1} and analyze which of the available choices lead to valid second minimal odd orbits.	
	\small{
	\begin{equation}
		\left(\begin{array}{ccccccccc}   
			1   & 2    & 3  & \cdots & k+1 & \cdots & 2k-1 & 2k & 2k+1 \\ 
			\langle k+1, & \begin{matrix} 2k+1\\ 2k \end{matrix}\rangle & \begin{matrix} 2k\\ 2k+1 \end{matrix} & \cdots & k+2 & \cdots & \begin{matrix} 2\\ 3 \end{matrix} & 1  & \begin{matrix} 3\\ 2 \end{matrix} \\
		\end{array} \right)	
		\label{eq:case2-2km1}	
	\end{equation}}	
	
	\begin{enumerate}
		\item If $f(1)=2k \Rightarrow 2$-suborbit $\left \{ 1, 2k \right \}$, a contradiction.
		\item If $f(1)=2k+1 \Rightarrow f(2)=k+1, f(3)=2k$
		\begin{enumerate}
			\item If $f(2k-1)=2 \Rightarrow f(2k+1)=3 \Rightarrow 4$-suborbit $\left \{ 1,3, 2k, 2k+1 \right \}$, a contradiction.
			\item If $f(2k-1)=3 \Rightarrow f(2k+1)=2$, by we have a second minimal $(2k+1)$ orbit given in \eqref{eq:case1-2km1-1} with topological structure min-max-min, observe that this is the same as \eqref{eq:case1-2km1-1} from setting $(1,2k-1)$, and so the settings share a cyclic permutation. This is expected since to move from setting $(1, 2k-1)$ to $(2,2k-1)$, only the location of $\tilde{J}$ is changed and so, in this particular case, the digraph remains unchanged as we simply swap the intervals $\tilde{J}$ and $J_{r_{2k-2}}$.
			\end{enumerate}
		\item $f(1)=k+1, f(2)=2k \Rightarrow f(3)=2k+1, f(2k+1)=2, f(2k-1)=3.$ This leads to a valid second minimal orbit with the topological structure max-min, the associated digraph is presented in \ref{fig:case2-2km1-2} and the cyclic permutation is listed in \eqref{eq:case2-2km1-2}. Next we prove that there are no primitive cycles of even length $\leq 2k-2$, and therefore according to Lemma~\ref{thm:convStraffin} $P$-linearization of the orbit \eqref{eq:case2-2km1-2} presents an example of continuous map with second minimal $(2k+1)$-orbit.
		
		\begin{enumerate}
			\item Consider primitive cycles that contain $J_{1}$. Without loss of generality choose $J_1$ as the starting vertex. First assume that cycle doesn't contain $J_{k+1}$. Since  $J_{2k+1-i}{\color{Red}\rightarrow}J_{i},\ i=1,...,k-1; i\neq 2;$ and $J_{2k-1} \rightarrow J_1, J_{2k-1} \rightarrow J_2$  any such cycle can be formed only by adding to starting vertex $J_1$ pairs $(J_{2k-1},J_2)$, $( J_{2k+1-i}, J_{i}), i=1,...,k-1$. Therefore, the length of the cycle (by counting $J_1$ twice) will be always an odd number. On the contrary, if cycle contains $J_{k+1}$, then to close it at $J_1$ the smallest required even length is $2k$.
			
			\item Consider primitive cycles that doesn't contain $J_{1}$. It is easy to see that by excluding $J_1$ from the cycle, due to red edges all the vertices but $J_{k+1}$ must be also excluded, and cycle at $J_{k+1}$ is the only possibility. 

		\end{enumerate}		
		\item $f(1)=k+1, f(2)=2k+1 \Rightarrow f(3)=2k, f(2k+1)=3, f(2k-1)=2$. This produces a second minimal $(2k+1)$ orbit with topological structure max-min, the associated digraph is presented in \ref{fig:case2-2km1-3} and the cyclic permutation is listed in \eqref{eq:case2-2km1-3}. As in previous cases we prove that there are no primitive cycles of even length $\leq 2k-2$, and therefore according to Lemma~\ref{thm:convStraffin} $P$-linearization of the orbit \eqref{eq:case2-2km1-3} presents an example of continuous map with second minimal $(2k+1)$-orbit.

			\begin{enumerate}
			\item Consider primitive cycles that contain $J_{1}$. Without loss of generality $J_1$ as the starting vertex. First assume that cycle doesn't contain $J_{k+1}$. Since  $J_{2k+1-i}{\color{Red}\rightarrow}J_{i},\ i=4,...,k-1$ (if $k\geq 5$) and $J_{2k-2} \rightarrow J_3$, $J_{2k-2} \rightarrow J_2$, $J_{2k-1} {\color{Red}\rightarrow}J_1$, $J_{2k} \rightarrow J_1$  any such cycle can be formed only by adding to starting vertex $J_1$ pairs $(J_{2k-1},J_1)$, $(J_{2k}, J_2)$, $(J_{2k-2}, J_2)$, $(J_{2k+1-i}, J_{i}), i=1,...,k-1; i \neq 2;$ Therefore, the length of the cycle (by counting $J_1$ twice) will be always an odd number. On the contrary, if cycle contains $J_{k+1}$, then to close it at $J_1$ the smallest required even length is $2k$.
			
			\item Consider primitive cycles that doesn't contain $J_{1}$. It is easy to see that by excluding $J_1$ from the cycle, due to red edges all the vertices but $J_{k+1}, J_{2k}, J_2$ must be also excluded, and cycle at $J_{k+1}$ and cycle formed by $J_2$ and $J_{2k}$ are only possibilities. 		
			
		\end{enumerate}
	\end{enumerate}
\end{proof}

\begin{lem}
	Placing $\tilde{J}$ and $\hat{J}$ in setting $(2,2k-2)$ produces exactly $4$ second minimal orbits. These are listed in cyclic permutations \eqref{eq:case1-2km1-2}, \eqref{eq:case2-2km1-2}, \eqref{eq:case2-2km2-3}, and \eqref{eq:case2-2km2-4} and the corresponding digraphs are presented in Fig.~\ref{fig:case1-2km1-2}, Fig.~\ref{fig:case2-2km1-2}, Fig.~\ref{fig:case2-2km2-3}, and Fig.~\ref{fig:case2-2km2-4} respectively.

	\small{
	\begin{equation}
		\left(\begin{array}{cccccccccccc}   
			1   & 2  & 3    & 4    & \cdots & k+1 & k+2 & \cdots & 2k-2 & 2k-1 & 2k & 2k+1 \\ 
			k+1 & 2k & 2k+1 & 2k-1 & \cdots & k+2 & k   & \cdots & 4    & 2    & 3  & 1 \\
		\end{array} \right)	
		\label{eq:case2-2km2-3}
	\end{equation}}			
	
	\small{
	\begin{equation}
		\left(\begin{array}{cccccccccccc}   
			1   & 2    & 3    & 4  & 5    & \cdots & k+1 & k+2 & \cdots & 2k-1 & 2k & 2k+1 \\ 
			k+1 & 2k-1 & 2k+1 & 2k & 2k-2 & \cdots & k+2 & k   & \cdots & 3    & 2  & 1 \\
		\end{array} \right)	
		\label{eq:case2-2km2-4}
	\end{equation}}		

  \input{digraphs/digraph2-2km2-3.tex}
	
	\input{digraphs/digraph2-2km2-4.tex}
	
	\label{lem:case2-2km2}
\end{lem}
\begin{proof}
Having inserted $(\tilde{J}, \hat{J})$ in position $(2,2k-2)$ and by using the path \eqref{eq:secondminfc_index} and the rules in \eqref{eq:2minrulesForward}, and \eqref{eq:2minrulesBackward} we observe that the images of the elements of the cycle from $5$ to $2k-2$ are uniquely defined following Stefan structure as it is demonstrated in \eqref{eq:case1-2km1-2-5-2k-2}.	 
		\small{
		\begin{equation}
			\left(\begin{array}{cccccccccc}   
				5  & \cdots & k & k+1 & k+2 & \cdots & 2k-2 \\ 
				2k-2 & \cdots & k+3 & k+2 & k 	 & \cdots & 4  \\
			\end{array} \right)	
			\label{eq:case1-2km1-2-5-2k-2}	
		\end{equation}}
The alteration appears only in images of the elements $1,2,4,2k-1,2k, 2k+1$. By using the path \eqref{eq:secondminfc_index} and the rules in \eqref{eq:2minrulesForward}, and \eqref{eq:2minrulesBackward} again, we construct the potential cyclic permutation \eqref{eq:cycpermi2kmi2}and analyze which of the available choices lead to valid second minimal odd orbits.	
	\small{
	\begin{equation}
		\left(\begin{array}{cccccccc}   
			1   & 2    & 3  & 4 & \cdots & 2k-1 & 2k & 2k+1 \\ 
			\langle k+1, & \begin{matrix} 2k-1\\ 2k \end{matrix}\rangle & 2k+1 & \begin{matrix} 2k\\ 2k-1 \end{matrix} & \cdots & \begin{matrix} 2\\ 3 \end{matrix} & \langle 1,  & \begin{matrix} 3\\ 2 \end{matrix} \rangle\\
		\end{array} \right)	
		\label{eq:cycpermi2kmi2}
	\end{equation}}	
	
	\begin{enumerate}
		\item If $f(1)=2k \Rightarrow f(2)=k+1, f(4)=2k-1$
		\begin{enumerate}
			\item If $f(2k)=1 \Rightarrow 2$-suborbit $\left \{ 1, 2k \right \}$, a contradiction.
			\item If $f(2k)=2$, we have a second minimal $(2k+1)$ orbit given in \eqref{eq:case1-2km1-2} with topological structure min-max, shared with setting $(1,2k-1)$. 
			\item If $f(2k)=3\Rightarrow 4$-suborbit $\left \{ 1,3, 2k, 2k+1 \right \}$, a contradiction.
		\end{enumerate}
		\item If $f(1)=2k-1\Rightarrow f(2)=k+1, f(4)=2k$
		\begin{enumerate}
			\item If $f(2k)=1$ and $f(2k-1)=2\Rightarrow f(2k+1)=3\Rightarrow 2$-suborbit $\left \{ 3, 2k+1 \right \}$, a contradiction.
			\item If $f(2k)=1$ and $f(2k-1)=3\Rightarrow f(2k+1)=2$, then for $k>3$ we have the primitive subgraph
					$$J_{r_{1}}\rightarrow J_{r_{1}}\rightarrow\dots\rightarrow J_{r_{2k-6}}\rightarrow \hat{J}\rightarrow J_{r_{2k-2}}\rightarrow J_{r_{1}}$$
					Lemma~\ref{thm:straffin} implies the existence of $2k-3$-periodic orbit, which is a contradiction. For $k=3$ we have the subgraph
					$$J_{r_{1}}\rightarrow \hat{J}\rightarrow J_{r_{4}}\rightarrow J_{r_{1}}$$
					which leads to a $3$-orbit, a contradiction.
					
			\item If $f(2k)=2\Rightarrow f(2k-1)=3, f(2k+1)=1\Rightarrow 4$-suborbit $\left \{ 1,3, 2k-1, 2k+1 \right \}$, a contradiction.
			\item If $f(2k)=3\Rightarrow f(2k-1)=2, f(2k+1)=1$, then for $k>3$ we have the subgraph
					$$J_{r_1}\rightarrow J_{r_{1}}\rightarrow\dots\rightarrow J_{r_{2k-6}}\rightarrow \hat{J}\rightarrow \tilde{J}\rightarrow J_{r_{1}}$$
					Lemma~\ref{thm:straffin} implies the existence of $2k-3$-periodic orbit, which is a contradiction. For $k=3$ we have the subgraph

					$$J_{r_{1}}\rightarrow \hat{J}\rightarrow \tilde{J}\rightarrow J_{r_{1}}$$
					which leads to a $3$-orbit, a contradiction.
		\end{enumerate}
		\item If $f(2)=2k\Rightarrow f(1)=k+1, f(4)=2k-1$
		\begin{enumerate}
			\item If $f(2k)=1$ and $f(2k-1)=2\Rightarrow f(2k+1)=3 \Rightarrow 2$-suborbit $\left \{ 3, 2k+1 \right \}$, a contradiction.
			\item If $f(2k)=1$ and $f(2k-1)=3\Rightarrow f(2k+1)=2$, we have a second minimal $(2k+1)$ orbit given in \eqref{eq:case2-2km1-2} with topological structure max-min, shared with setting $(2,2k-1)$. 
			\item If $f(2k)=2\Rightarrow 2$-suborbit $\left \{ 2, 2k \right \}$, a contradiction.
			\item If $f(2k)=3\Rightarrow f(2k-1)=2, f(2k+1)=1$, we have a second minimal $(2k+1)$ orbit given in \eqref{eq:case2-2km2-3} with topological structure max-min-max, and the associated digraph is presented in Fig.~\ref{fig:case2-2km2-3}. Next we prove as in previous lemma that there are no primitive cycles of even length $\leq 2k-2$, and therefore according to Lemma~\ref{thm:convStraffin} $P$-linearization of the orbit \eqref{eq:case2-2km2-3} presents an example of continuous map with second minimal $(2k+1)$-orbit.
		
		\begin{enumerate}
			\item Consider primitive cycles that contain $J_{1}$. Without loss of generality $J_1$ as the starting vertex. First assume that cycle doesn't contain $J_{k+1}$. Since  $J_{2k+1-i}{\color{Red}\rightarrow}J_{i},\ i=2,...,k-1; i\neq 3;$ and $J_{2k} \rightarrow J_1, J_{2k-2} \rightarrow J_3$,  $J_{2k-2} \rightarrow J_2$ any such cycle can be formed only by adding to starting vertex $J_1$ pairs $(J_{2k-2},J_2)$, $( J_{2k+1-i}, J_{i}), i=1,...,k-1$. Therefore, the length of the cycle (by counting $J_1$ twice) will be always an odd number. On the contrary, if cycle contains $J_{k+1}$, then to close it at $J_1$ the smallest required even length is $2k$.
			
		\item Consider primitive cycles that doesn't contain $J_{1}$. It is easy to see that by excluding $J_1$ from the cycle, due to red edges all the vertices but $J_{k+1}, J_2, J_{2k}$ must be also excluded, and cycle at $J_{k+1}$ and cycle formed by $J_2$ and $J_{2k}$ are the only possibilities. 			
		\end{enumerate}
		\end{enumerate}
		\item If $f(2)=2k-1\Rightarrow f(1)=k+1, f(4)=2k$
		\begin{enumerate}
			\item If $f(2k)=1$ and $f(2k-1)=2 \Rightarrow 2$-suborbit $\left \{ 2, 2k-1 \right \}$, a contradiction.
			\item If $f(2k)=1$ and $f(2k-1)=3 \Rightarrow 4$-suborbit $\left \{ 2, 3, 2k-1, 2k+1 \right \}$, a contradiction.
			\item If $f(2k)=2\Rightarrow f(2k-1)=3, f(2k+1)=1$, we have a second minimal $(2k+1)$ orbit given in \eqref{eq:case2-2km2-4} with topological structure max-min-max, and the associated digraph is presented in Fig.~\ref{fig:case2-2km2-4}. Next we prove as in previous cases that there are no primitive cycles of even length $\leq 2k-2$, and therefore according to Lemma~\ref{thm:convStraffin} $P$-linearization of the orbit \eqref{eq:case2-2km2-4} presents an example of continuous map with second minimal $(2k+1)$-orbit.
		
		\begin{enumerate}
			\item Consider primitive cycles that contain $J_{1}$. Without loss of generality choose $J_1$ as the starting vertex. First assume that cycle doesn't contain $J_{k+1}$. Since  $J_{2k+1-i}{\color{Red}\rightarrow}J_{i},\ i=1,...,k-1;$ any such cycle can be formed only by adding to starting vertex $J_1$ pairs $( J_{2k+1-i}, J_{i}), i=1,...,k-1$. Therefore, the length of the cycle (by counting $J_1$ twice) will be always an odd number. On the contrary, if cycle contains $J_{k+1}$, then to close it at $J_1$ the smallest required even length is $2k$.
			
		\item Consider primitive cycles that doesn't contain $J_{1}$. It is easy to see that by excluding $J_1$ from the cycle, due to red edges all the vertices but $J_{k+1}, J_2, J_{2k-1}$ must be also excluded, and cycle at $J_{k+1}$ and cycle formed by $J_2$ and $J_{2k-1}$ are the only possibilities. 			
		\end{enumerate}

			\item If $f(2k)=3\Rightarrow f(2k-1)=2\Rightarrow 2$-suborbit $\left \{ 2, 2k-1 \right \}$, a contradiction.
		\end{enumerate}
	\end{enumerate}
\end{proof}
		

	
		\begin{lem}
		Each setting $(i,j)$ with $2< i < k$ and $j = 2k-i, k>3$ produces exactly $4$ second minimal cycles listed in cyclic permutations \eqref{eq:validij1}, \eqref{eq:validij2}, \eqref{eq:validij3}, and \eqref{eq:validij4}. If $i=3$, \eqref{eq:validij1} repeats the cyclic permutation \eqref{eq:case2-2km2-4} revealed in Lemma~\ref{lem:case2-2km2}. When $i=k-1$, the cyclic permutation \eqref{eq:validij4} and \eqref{eq:len2kcase11cyc} from Lemma~\ref{lem:m2km1} are identical.
 The corresponding digraphs are presented in Figures~\ref{fig:validij1}, \ref{fig:validij2}, \ref{fig:validij3}, and \ref{fig:validij4} respectively, when $3<i<k-1$.
		
		\small{
		\begin{equation}
			\left(\begin{array}{cccccccccccc}   
				1   & \cdots & i-1 & i   & i+1 & i+2 & \cdots & j+1 & j+2 & j+3 & \cdots & 2k+1 \\ 
				k+1 & \cdots & j+2 & j+4 & j+3 & j+1 & \cdots & i+1 & i   & i-1 & \cdots & 1 \\
			\end{array} \right)
			\label{eq:validij1}
		\end{equation}}		
			
		\small{
		\begin{equation}
			\left(\begin{array}{cccccccccccc}   
				1   & \cdots & i-1 & i   & i+1 & i+2 & \cdots & j+1 & j+2 & j+3 & \cdots & 2k+1  \\ 
				k+1 & \cdots & j+4 & j+2 & j+3 & j+1 & \cdots & i   & i+1 & i-1 & \cdots & 1 \\
			\end{array} \right)
			\label{eq:validij2}
		\end{equation}}				
			
		\small{
		\begin{equation}
			\left(\begin{array}{cccccccccccc}   
				1   & \cdots & i-1 & i   & i+1 & i+2 & \cdots & j+1 & j+2 & j+3 & \cdots & 2k+1  \\ 
				k+1 & \cdots & j+4 & j+2 & j+3 & j+1 & \cdots & i+1 & i-1 & i   & \cdots & 1 \\
			\end{array} \right)
			\label{eq:validij3}
		\end{equation}}				
		
		\small{
		\begin{equation}
			\left(\begin{array}{cccccccccccc}   
				1   & \cdots & i-1 & i   & i+1 & i+2 & \cdots & j+1 & j+2 & j+3 & \cdots & 2k+1  \\ 
				k+1 & \cdots & j+4 & j+1 & j+3 & j+2 & \cdots & i+1 & i   & i-1 & \cdots & 1 \\
			\end{array} \right)
			\label{eq:validij4}
		\end{equation}}	
		
		\input{digraphs/digraph-ij1.tex}
		
		\input{digraphs/digraph-ij2.tex}		
		
		\input{digraphs/digraph-ij3.tex}			
		
		\input{digraphs/digraph-ij4.tex}						
				
		\label{lem:settingij}
	\end{lem}
	\begin{proof}
		We prove this by doing a case by case analysis of the general cyclic permutation listed in \eqref{eq:cycpermi2kmi1}.
		\small{
		\begin{equation}
			\left(\begin{array}{cccccccccccc}   
				\cdots & i-1 & i & i+1 & i+2 & \cdots & j+1 & j+2 & j+3 & \cdots \\ 
				\cdots & < \begin{matrix} j+2\\ j+1 \end{matrix}, & j+4 > & j+3 & \begin{matrix} j+1\\ j+2 \end{matrix} & \cdots & \begin{matrix} i+1\\ i \end{matrix} & < \begin{matrix} i\\ i+1 \end{matrix}, & i-1 > & \cdots \\
			\end{array} \right)
			\label{eq:cycpermi2kmi1}
		\end{equation}}	
					
		\begin{enumerate}
			\item If $f(i-1)=j+2\Rightarrow f(i) = j+4, f(i+2)=j+1$
			\begin{enumerate}
				\item If $f(j+2)=i-1\Rightarrow 2$-suborbit $\left \{ i-1, j+2 \right \}$, a contradiction.
				\item If $f(j+2)=i+1\Rightarrow 4$-suborbit $\left \{ i-1, i+1, j+2, j+3 \right \}$, a contradiction.
				\item If $f(j+2)=i$, we have a second minimal $(2k+1)$ orbit given in \eqref{eq:validij1} with topological structure max-min-max, provided $i>3$. If $i=3$ then we have a second minimal $2k+1$-orbit \eqref{eq:case2-2km2-4} with topological structure {\it max} revealed in Lemma~\ref{lem:case2-2km2}. 
			\end{enumerate}
			\item If $f(i-1)=j+1\Rightarrow f(i) = j+4, f(i+2)=j+2$
			\begin{enumerate}
				\item If $f(j+1)=i+1\Rightarrow 4$-suborbit $\left \{ i-1, i+1, j+1, j+3 \right \}$, a contradiction.
				\item If $f(j+1)=i\Rightarrow f(j+3)\neq i+1$ or we have the closed $2$-suborbit $\left \{ i+1, j+3 \right \}$. So we must have $f(j+3)=i-1\Rightarrow f(j+2)=i+1$. Following the proof of the Lemma~\ref{lem:JhatJtilde} it follows that for $2<i<k-1$ the digraph of the cyclic permutation contains a primitive subgraph 
				
				$$J_{r_{1}}\rightarrow J_{r_{2}}\rightarrow\dots\rightarrow J_{r_{w-4}}\rightarrow \hat{J}\rightarrow \tilde{J}\rightarrow J_{r_{w+1}}\rightarrow  \dots \rightarrow J_{r_{2k-2}}\rightarrow J_{r_{1}}\rightarrow J_{r_{1}}$$
				
				and for $i=k-1$ the digraph of the cyclic permutation contains a primitive subgraph
				
				$$J_{r_{1}}\rightarrow \hat{J}\rightarrow \tilde{J}\rightarrow J_{r_{5}}\rightarrow \dots \rightarrow J_{r_{2k-2}}\rightarrow J_{r_{1}}$$
				
				both of which have length $2k-2$. By Lemma~\ref{thm:straffin}, a periodic orbit of period $2k-3$ must exist, which is a contradiction.
				
				\item If $f(j+1)=i+1\Rightarrow f(j+3)\neq i-1$or there is a period $4$-suborbit $i-1, i+1$, $j+1, j+3$. Thus, $f(j+3)=i\Rightarrow f(j+2)=i-1$. By repeating the argument of the previous case we prove the existence of the $2k-3$-orbit, which is a contradiction. 
				
			\end{enumerate}
			\item If $f(i)=j+2\Rightarrow f(i-1)=j+4, f(i+2)=j+1$
			\begin{enumerate}
				\item If $f(j+2)=i \Rightarrow$ closed $2$-suborbit $\left \{ i, j+2 \right \}$, a contradiction.
				\item If $f(j+2)=i+1\Rightarrow f(j+1)=i, f(j+3)=i-1$, we have a second minimal $(2k+1)$ orbit given in \eqref{eq:validij2} with topological structure max-min-max-min-max.
				\item If $f(j+2)=i-1$ and $f(j+3)=i+1\Rightarrow$ closed $2$-suborbit $\left \{ i+1, j+3 \right \}$, a contradiction. 
				\item If $f(j+2)=i-1, f(j+3)=i\Rightarrow f(j+1)=i+1$, we have a second minimal $(2k+1)$ orbit given in \eqref{eq:validij3} with topological structure max-min-max-min-max. 
			\end{enumerate}
			\item If $f(i)=j+1\Rightarrow f(i-1)=j+4, f(i+2)=j+2$
			\begin{enumerate}
				\item If $f(j+1)=i\Rightarrow 2$-suborbit $\left \{ i, j+1 \right \}$, a contradiction.
				\item If $f(j+1)=i+1$ and $f(j+3)=i\Rightarrow 4$-suborbit $\left \{ i, i+1, j+1, j+3 \right \}$, a contradiction.
				\item If $f(j+1)=i+1$ and $f(j+3)=i-1$, we have a second minimal $(2k+1)$ orbit given in \eqref{eq:validij4} with topological structure max-min-max, unless $i=3$, in which case the topological structure is single max. When $i=k-1$, the cyclic permutation \eqref{eq:validij4} repeats the cyclic permutation \eqref{eq:len2kcase11cyc} revealed in  Lemma~\ref{lem:m2km1}.
			\end{enumerate}
		\end{enumerate}
		
Observe, as $i$ varies between $3$ and $k-1$, the structure of the digraphs associated with the cyclic permutations changes. In particular, for a given cyclic permutation, varying $i$ from $3$ to $k-1$ shifts the region of variations from the right to left ends of the digraph. We demonstrate this in Figures~\ref{fig:ij1} through \ref{fig:ij4}. Note that in these subgraphs, with the exception of Fig.~\ref{fig:ij1a} where $J_{1}\not\rightarrow J_{2k-1}, J_{2k}$, we have $J_{1}\rightarrow J_{k+1},\dots,J_{2k}$.

\begin{figure}[htb]
	\centering		
	\subcaptionbox{$i=3$\label{fig:ij1a}}{\input{digraphs/subdigraph-ij-1_1.tex}} 					
	\subcaptionbox{$3 < i < k-1$}{\input{digraphs/subdigraph-ij-1_2.tex}}	
	\subcaptionbox{$i = k-1$}{\input{digraphs/subdigraph-ij-1_3.tex}}
	\caption{Portion with variations in digraphs of cyclic permutation \ref{eq:validij1}, $(\tilde{J},\hat{J})$ in setting $(i,j), 3 \leq i \leq k-1$.}
	\label{fig:ij1}
\end{figure}

\begin{figure}[htb]
	\centering		
	\subcaptionbox{$i=3$}{\input{digraphs/subdigraph-ij-2_1.tex}} 					
	\subcaptionbox{$3 < i < k-1$}{\input{digraphs/subdigraph-ij-2_2.tex}}	
	\subcaptionbox{$i = k-1$\label{fig:ij2c}}{\input{digraphs/subdigraph-ij-2_3.tex}}
	\caption{Portion with variations in digraphs of cyclic permutation \ref{eq:validij2}, $(\tilde{J},\hat{J})$ in setting $(i,j), 3 \leq i \leq k-1$.}
	\label{fig:ij2}
\end{figure}

\begin{figure}[htb]
	\centering		
	\subcaptionbox{$i=3$}{\input{digraphs/subdigraph-ij-3_1.tex}} 					
	\subcaptionbox{$3 < i < k-1$}{\input{digraphs/subdigraph-ij-3_2.tex}}	
	\subcaptionbox{$i = k-1$}{\input{digraphs/subdigraph-ij-3_3.tex}}
	\caption{Portion with variations in digraphs of cyclic permutation \ref{eq:validij3}, $(\tilde{J},\hat{J})$ in setting $(i,j), 3 \leq i \leq k-1$.}
	\label{fig:ij3}
\end{figure}

\begin{figure}[htb]
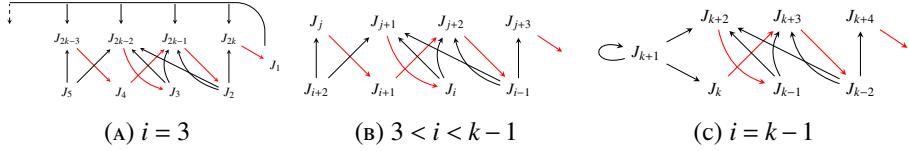

	\centering		
	\subcaptionbox{$i=3$}{\input{digraphs/subdigraph-ij-4_1.tex}} 					
	\subcaptionbox{$3 < i < k-1$}{\input{digraphs/subdigraph-ij-4_2.tex}}	
	\subcaptionbox{$i = k-1$\label{fig:ij4c}}{\input{digraphs/subdigraph-ij-4_3.tex}}
	\caption{Portion with variations in digraphs of cyclic permutation \ref{eq:validij4}, $(\tilde{J},\hat{J})$ in setting $(i,j), 3 \leq i \leq k-1$.}
	\label{fig:ij4}
\end{figure}

Note that all four cyclic permutations are simple. Next we analyze the digraphs to show that there are no primitive cycles of even length $\leq 2k-2$, which would imply by Straffin's lemma an existence of odd periodic orbit of length $\leq 2k-3$. From Lemms~\ref{thm:convStraffin} it then follows that the the $P$-linearization of the orbits \eqref{eq:validij1}, \eqref{eq:validij2}, \eqref{eq:validij3}, and \eqref{eq:validij4} present an example of continuous map with second minimal $(2k+1)$-orbit. The proof coincides with the similar proofs given in previous lemmas. 

		\begin{enumerate}
			\item Consider primitive cycles that contain $J_{1}$. Without loss of generality choose a starting vertex as $J_1$. First assume that cycle doesn't contain $J_{k+1}$. Due to presence of red edges any such cycle can be formed only by successfully adding to starting vertex $J_1$ pairs $(J_{p}, J_q)$, where  $q\in \{1,...,k-1\}$, $p\in \{k+2,...,2k\}$. Therefore, the length of the cycle (by counting $J_1$ twice) will be always an odd number. On the contrary, if cycle contains $J_{k+1}$, then besides the new pair $(J_{k+1}, J_k)$ or $(J_{k+1}, J_{k-1})$ there is a possibility to add just $J_{k+1}$ alone due to loop at $J_{k+1}$, and hence to build a primitive subgraph of even length. However, the smallest required even length is $2k$, and therefore no odd orbits of period smaller than $2k-1$ can be produced. 
			
			\item Consider primitive cycles that doesn't contain $J_{1}$. Since $J_1 \rightarrow J_{k+1}$ and $J_{k+1}\rightarrow J_{k+1}$ are only edges directed to $J_{k+1}$, we have to exclude $J_{k+1}$ from the primitive cycle unless it is a loop at $J_{k+1}$. But then any primitive cycle formed by the remaining intervals can be formed by adding some  of the indicated pairs $(J_{p}, J_q)$ to starting vertex, and therefore all are of odd length.
			 			\end{enumerate}

	\end{proof}
	
	\begin{lem}
		Each setting $(i,j+1)$ with $2< i <k$ and $j = 2k-i, k>3$ produces exactly $4$ second minimal cycles listed in cyclic permutations \eqref{eq:validij11}, \eqref{eq:validij12}, \eqref{eq:validij13}, and \eqref{eq:validij14}. Cyclic permutation \eqref{eq:validij12} repeats \eqref{eq:validij3} from Lemma~\ref{lem:settingij}.  If $i=3$, \eqref{eq:validij11} repeats the cyclic permutation \eqref{eq:case2-2km2-3}, revealed in Lemma~\ref{lem:case2-2km2}, and \eqref{eq:validij14} repeats the cyclic permutation \eqref{eq:case2-2km1-3}, revealed in Lemma~\ref{lem:case2-2km1}. The corresponding digraphs are presented in Figures~\ref{fig:validij11}, \ref{fig:validij12}, \ref{fig:validij13}, and \ref{fig:validij14} respectively.
		
		\small{
		\begin{equation}
			\left(\begin{array}{ccccccccccc}   
				1   & \cdots & i-1 & i   & i+1 & \cdots & j+1 & j+2 & j+3 & j+4 & \cdots \\ 
				k+1 & \cdots & j+3 & j+4 & j+2 & \cdots & i+1 & i-1 & i   & i-2 & \cdots \\
			\end{array} \right)
			\label{eq:validij11}
		\end{equation}}			
		
		\small{
		\begin{equation}
			\left(\begin{array}{ccccccccccc}   
				1   & \cdots & i-1 & i   & i+1 & \cdots & j+1 & j+2 & j+3 & j+4 & \cdots \\ 
				k+1 & \cdots & j+4 & j+2 & j+3 & \cdots & i+1 & i-1 & i   & i-2 & \cdots \\
			\end{array} \right)
			\label{eq:validij12}
		\end{equation}}		
			
		\small{
		\begin{equation}
			\left(\begin{array}{ccccccccccc}   
				1   & \cdots & i-1 & i   & i+1 & \cdots & j+1 & j+2 & j+3 & j+4 & \cdots \\ 
				k+1 & \cdots & j+4 & j+3 & j+2 & \cdots & i   & i-1 & i+1 & i-2 & \cdots \\
			\end{array} \right)
			\label{eq:validij13}
		\end{equation}}				
			
		\small{
		\begin{equation}
			\left(\begin{array}{ccccccccccc}   
				1   & \cdots & i-1 & i   & i+1 & \cdots & j+1 & j+2 & j+3 & j+4 & \cdots \\ 
				k+1 & \cdots & j+4 & j+3 & j+2 & \cdots & i+1 & i-1 & i-2 & i   & \cdots \\
			\end{array} \right)
			\label{eq:validij14}
		\end{equation}}				
		
		\input{digraphs/digraph-ij1-1.tex}		
		
		\input{digraphs/digraph-ij1-2.tex}	
		
		\input{digraphs/digraph-ij1-3.tex}	
		
		\input{digraphs/digraph-ij1-4.tex}			
		
		\label{lem:settingijp1}
	\end{lem}
	\begin{proof}
		We prove this by doing a case by case analysis of the general cyclic permutation listed in \eqref{eq:cycpermi2kmip1}.
		\small{
		\begin{equation}
			\left(\begin{array}{cccccccccc}   
				\cdots & i-1 & i & i+1 & \cdots & j+1 & j+2 & j+3 & j+4 & \cdots \\ 
				\cdots & < \begin{matrix} j+2\\ j+3 \end{matrix}, & j+4 > & \begin{matrix} j+3\\ j+2 \end{matrix} & \cdots & \begin{matrix} i+1\\ i \end{matrix} & i-1 & < \begin{matrix} i\\ i+1 \end{matrix}, & i-2 > & \cdots \\
			\end{array} \right)
			\label{eq:cycpermi2kmip1}
		\end{equation}}		
		
		\begin{enumerate}
			\item If $f(i-1)=j+2\Rightarrow 2$-suborbit $\left \{ i-1, j+2 \right \}$, a contradiction.
			\item If $f(i-1)=j+3\Rightarrow f(i)=j+4, f(i+1)=j+2$
			\begin{enumerate}
				\item If $f(j+3)=i\Rightarrow f(j+1)=i+1, f(j+4)=i-2$, we have a second minimal $(2k+1)$ orbit given in \eqref{eq:validij11} with topological structure max-min-max-min-max, if $i>3$. If $i=3$ then we have a second minimal $2k+1$-orbit \eqref{eq:case2-2km2-3} with topological structure max-min-max revealed in Lemma~\ref{lem:case2-2km2}. 
				\item If $f(j+3)=i+1\Rightarrow 4$-suborbit $\left \{ i-1, i+1, j+2, j+3 \right \}$, a contradiction.
				\item If $f(j+3)=i-2$ and $f(j+4)=i\Rightarrow 2$-suborbit $\left \{ i, j+4 \right \}$, a contradiction.
				\item If $f(j+3)=i-2$ and $f(j+4)=i+1\Rightarrow f(j+1)=i$.  Following the proof of the Lemma~\ref{lem:JhatJtilde} it follows that for $2<i<k-1$ the digraph of the cyclic permutation contains a primitive subgraph 
				
				$$J_{r_{1}}\rightarrow J_{r_{2}}\rightarrow\dots\rightarrow J_{r_{w-3}}\rightarrow\tilde{J}\rightarrow\hat{J}\rightarrow J_{r_{w+2}}\rightarrow J_{r_{w+3}}\rightarrow\dots\rightarrow J_{r_{2k-2}}\rightarrow J_{r_{1}}\rightarrow J_{r_{1}}$$
				
				and for $i=k-1$ the digraph of the cyclic permutation contains a primitive subgraph
				
				$$J_{r_{1}}\rightarrow \tilde{J}\rightarrow \hat{J}\rightarrow J_{r_{6}}\rightarrow \dots \rightarrow J_{r_{2k-2}}\rightarrow J_{r_{1}}\rightarrow J_{r_{1}}$$
				
				both of which have length $2k-2$. By Lemma~\ref{thm:straffin}, a periodic orbit of period $2k-3$ must exist, which is a contradiction.
			\end{enumerate}
			\item If $f(i)=j+2\Rightarrow f(i-1)=j+4, f(i+1)=j+3$
			\begin{enumerate}
				\item If $f(j+3)=i\Rightarrow f(j+1)=i+1, f(j+4)=i-2$, we have a second minimal $(2k+1)$ orbit given in \eqref{eq:validij12} with topological structure max-min-max-min-max. Cyclic permutation \eqref{eq:validij12} repeats \eqref{eq:validij3} from Lemma~\ref{lem:settingij}.
				\item If $f(j+3)=i+1\Rightarrow 2$-suborbit $\left \{ i+1, j+3 \right \}$, a contradiction.
				\item If $f(j+3)=i-2$ and $f(j+4)=i+1\Rightarrow f(j+1)=i$. This implies a cyclic permutation whose digraph contains primitive subgraph of length $2k-2$. The proof coincides with the proof given above in the case (2d). By Lemma~\ref{thm:straffin}, a periodic orbit of period $2k-3$ must exist, which is a contradiction.				
				\item If $f(j+3)=i-2$ and $f(j+4)=i\Rightarrow f(j+1)=i+1\Rightarrow 4$-suborbit $\left \{ i, i-1, j+2, j+4 \right \}$, a contradiction.
			\end{enumerate}
			\item If $f(i)=j+3\Rightarrow f(i-1)=j+4, f(i+1)=j+2$
			\begin{enumerate}
				\item If $f(j+3)=i \Rightarrow 2$-suborbit $\left \{ i, j+3 \right \}$, a contradiction.
				\item If $f(j+3)=i+1\Rightarrow f(j+1)=i, f(j+4)=i-2$, we have a second minimal $(2k+1)$ orbit given in \eqref{eq:validij13} with topological structure max-min-max.
				\item If $f(j+3)=i-2$ and $f(j+4)=i+1 \Rightarrow 4$-suborbit $\left \{ i-1, i+1, j+2, j+4 \right \}$, a contradiction.
				\item If $f(j+3)=i-2$ and $f(j+4)=i\Rightarrow f(j+1)=i+1$, we have a second minimal $(2k+1)$ orbit given in \eqref{eq:validij14} with topological structure max-min-max, if $i>3$. If $i=3$ then we have a second minimal $2k+1$-orbit \eqref{eq:case2-2km1-3} with topological structure max-min revealed in Lemma~\ref{lem:case2-2km1}.     
			\end{enumerate}
		\end{enumerate}
		
As $i$ varies between $3$ and $k-1$, the structure of the digraphs associated with the cyclic permutations changes. In particular, for a given cyclic permutation, varying $i$ from $3$ to $k-1$ shifts the region of variations from the right to left ends of the digraph. We demonstrate this in Fig.~\ref{fig:ij11} through Fig.~\ref{fig:ij14}. Observe that Fig.~\ref{fig:ij3} and Fig.~\ref{fig:ij12} are identical. Note that in these subgraphs, with the exception of Fig.~\ref{fig:ij11a} where $J_{1}\not\rightarrow J_{2k}$, we have $J_{1}\rightarrow J_{k+1},\dots,J_{2k}$.

\begin{figure}[htb]
	\centering		
	\subcaptionbox{$i=3$\label{fig:ij11a}}{\input{digraphs/subdigraph-ij1-1_1.tex}} 					
	\subcaptionbox{$3 < i < k-1$}{\input{digraphs/subdigraph-ij1-1_2.tex}}	
	\subcaptionbox{$i = k-1$}{\input{digraphs/subdigraph-ij1-1_3.tex}}
	\caption{Portion with variations in digraphs of cyclic permutation \ref{eq:validij11}, $(\tilde{J},\hat{J})$ in setting $(i,j), 3 \leq i \leq k-1$.}
	\label{fig:ij11}
\end{figure}

\begin{figure}[htb]
	\centering		
	\subcaptionbox{$i=3$}{\input{digraphs/subdigraph-ij1-2_1.tex}} 					
	\subcaptionbox{$3 < i < k-1$}{\input{digraphs/subdigraph-ij1-2_2.tex}}	
	\subcaptionbox{$i = k-1$}{\input{digraphs/subdigraph-ij1-2_3.tex}}
	\caption{Portion with variations in digraphs of cyclic permutation \ref{eq:validij12}, $(\tilde{J},\hat{J})$ in setting $(i,j), 3 \leq i \leq k-1$.}
	\label{fig:ij12}
\end{figure}

\begin{figure}[htb]
	\centering		
	\subcaptionbox{$i=3$}{\input{digraphs/subdigraph-ij1-3_1.tex}} 					
	\subcaptionbox{$3 < i < k-1$}{\input{digraphs/subdigraph-ij1-3_2.tex}}	
	\subcaptionbox{$i = k-1$}{\input{digraphs/subdigraph-ij1-3_3.tex}}
	\caption{Portion with variations in digraphs of cyclic permutation \ref{eq:validij13}, $(\tilde{J},\hat{J})$ in setting $(i,j), 3 \leq i \leq k-1$.}
	\label{fig:ij13}
\end{figure}

\begin{figure}[htb]
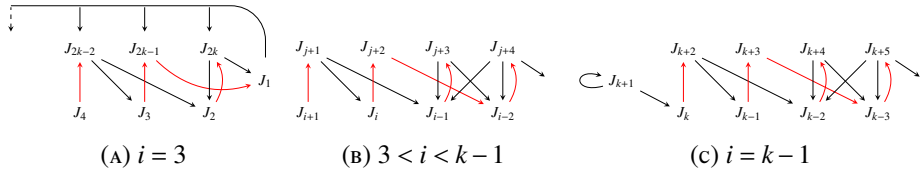

	\centering		
	\subcaptionbox{$i=3$\label{fig:ij14c}}{\input{digraphs/subdigraph-ij1-4_1.tex}} 					
	\subcaptionbox{$3 < i < k-1$}{\input{digraphs/subdigraph-ij1-4_2.tex}}	
	\subcaptionbox{$i = k-1$}{\input{digraphs/subdigraph-ij1-4_3.tex}}
	\caption{Portion with variations in digraphs of cyclic permutation \ref{eq:validij14}, $(\tilde{J},\hat{J})$ in setting $(i,j), 3 \leq i \leq k-1$.}
	\label{fig:ij14}
\end{figure}		
		
Note that all four cyclic permutations are simple. Finally, we aim to analyze the digraphs and demonstrate that there are no primitive cycles of even length $\leq 2k-2$, which would imply by Straffin's lemma an existence of odd periodic orbit of length $\leq 2k-3$. From Lemms~\ref{thm:convStraffin} it then follows that the the $P$-linearization of the orbits \eqref{eq:validij11}, \eqref{eq:validij12}, \eqref{eq:validij13}, and \eqref{eq:validij14} present an example of continuous map with second minimal $(2k+1)$-orbit. The proof coincides with the similar proof given in Lemma~\ref{lem:settingij}. 
	
	\end{proof}
	
	\begin{lem}
		Placing $(\tilde{J}, \hat{J})$ in relative positions $(k, k+1)$ produces exactly $4$ second minimal cycles listed in cyclic permutations \eqref{eq:validkk1}, \eqref{eq:validkk2}, \eqref{eq:validkk3}, and \eqref{eq:validkk4}. The corresponding digraphs are presented in Figures~\ref{fig:validkk1}, \ref{fig:validkk2}, \ref{fig:validkk3}, and \ref{fig:validkk4} respectively.
		
		\small{
		\begin{equation}
			\left(\begin{array}{ccccccccccc}   
				1   & 2 	 & \cdots & k-1 & k   & k+1 & k+2 & k+3 & k+4 & \cdots & 2k+1 \\ 
				k+1 & 2k+1 & \cdots & k+4 & k+2 & k+3 & k-1 & k   & k-2 & \cdots & 1\\
			\end{array} \right)
			\label{eq:validkk1}
		\end{equation}}
	
		\small{
		\begin{equation}
			\left(\begin{array}{ccccccccccccc}   
				1   & 2 	 & \cdots & k-1 & k   & k+1 & k+2 & k+3 & k+4 & k+5 & \cdots & 2k+1  \\ 
				k+1 & 2k+1 & \cdots & k+4 & k+3 & k+2 & k-1 & k-2 & k   & k-3 & \cdots & 1 \\
			\end{array} \right)
			\label{eq:validkk2}
		\end{equation}}			
		
		\small{
		\begin{equation}
			\left(\begin{array}{cccccccccccc}   
				1 & 2 	 & \cdots & k-1 & k   & k+1 & k+2 & k+3 & k+4 & \cdots & 2k+1  \\ 
				k & 2k+1 & \cdots & k+4 & k+3 & k+2 & k-1 & k+1 & k-2 & \cdots & 1 \\
			\end{array} \right)
			\label{eq:validkk3}
		\end{equation}}		
			
		\small{
		\begin{equation}
			\left(\begin{array}{cccccccccccc}   
				1   & 2 	 & \cdots & k-1 & k   & k+1 & k+2 & k+3 & k+4 & \cdots & 2k+1  \\ 
				k+1 & 2k+1 &\cdots & k+3 & k+4 & k+2 & k-1 & k   & k-2 & \cdots  & 1 \\
			\end{array} \right)
			\label{eq:validkk4}
		\end{equation}}	
		
		\input{digraphs/digraphk-k1-1.tex}	
	
		\input{digraphs/digraphk-k1-2.tex}	
	
		\input{digraphs/digraphk-k1-3.tex}	
	
		\input{digraphs/digraphk-k1-4.tex}		
	
		\label{lem:settingkk1}
	\end{lem}
	\begin{proof}
		We prove this by doing a case by case analysis of the general cyclic permutation listed in \eqref{eq:cycpermi2kmip1}. Note that in the frame of notation introduced in the proof of Lemma~\ref{lem:JhatJtilde} we have $w=2$;
		\small{
		\begin{equation}
			\left(\begin{array}{ccccccccc}   
				1 & \cdots & k-1 & k & k+1 & k+2 & k+3 & k+4 & \cdots \\ 
				\begin{matrix} k+1\\ k \end{matrix}  & \cdots & < k+4, & \begin{matrix} k+3\\ k+2 \end{matrix} > & \begin{matrix} k+2\\ k+3 \end{matrix} & k-1 & < k-2 & \begin{matrix} k\\ k+1 \end{matrix} > & \cdots \\
			\end{array} \right)
			\label{eq:cycpermi2kmipk}
		\end{equation}}			
		
		\begin{enumerate}
			\item If $f(k)=k+2\Rightarrow f(k-1)=k+4, f(k+1)=k+3$
			\begin{enumerate}
				\item If $f(k+4)=k\Rightarrow f(1)=k+1, f(k+3)=k-2 \Rightarrow 4$-suborbit $\left \{ k-1, k, k+2, k+4 \right \}$, a contradiction.
				\item If $f(k+4)=k+1\Rightarrow f(1)=k, f(k+3)=k-2$, then for $k>3$ we have the primitive subgraph
					$$\tilde{J}\rightarrow \hat{J}\rightarrow J_{r_{4}}\rightarrow\dots\rightarrow J_{r_{2k-2}}\rightarrow \tilde{J}$$
					of length $2k-2$. Lemma~\ref{thm:straffin} implies the existence of $2k-3$-orbit, which is a contradiction. 				\item If $f(k+4)=k-2$ and $f(k+3)=k$, we have a second minimal $(2k+1)$ orbit given in \eqref{eq:validkk1} with topological structure max-min-max-min-max.
				\item If $f(k+4)=k-2$ and $f(k+3)=k+1\Rightarrow 2$-suborbit $\left \{ k+1, k+3 \right \}$, a contradiction.
			\end{enumerate}
			\item If $f(k)=k+3\Rightarrow f(k-1)=k+4, f(k+1)=k+2$
			\begin{enumerate}
				\item If $f(k+4)=k\Rightarrow f(1)=k+1, f(k+3)=k-2$, we have a second minimal $(2k+1)$ orbit given in \eqref{eq:validkk2} with topological structure max-min-max. 
				\item If $f(k+4)=k+1\Rightarrow f(1)=k, f(k+3)=k-2 \Rightarrow 4$-suborbit $\left \{ k-1, k+1, k+2, k+4 \right \}$, a contradiction.
				\item If $f(k+4)=k-2$ and $f(k+3)=k \Rightarrow 2$-suborbit $\left \{ k, k+3 \right \}$, a contradiction.
				\item If $f(k+4)=k-2$ and $f(k+3)=k+1\Rightarrow f(1)=k$, we have a second minimal $(2k+1)$ orbit given in \eqref{eq:validkk3} with topological structure max-min-max.
			\end{enumerate}
			\item If $f(k-1)=k+2\Rightarrow 2$-suborbit $\left \{ k-1, k+2 \right \}$, a contradiction.
			\item If $f(k-1)=k+3\Rightarrow f(k)=k+4, f(k+1)=k+2$
			\begin{enumerate}
				\item If $f(k+4)=k \Rightarrow 2$-suborbit $\left \{ k, k+4 \right \}$, a contradiction.
				\item If $f(k+4)=k+1\Rightarrow f(1)=k, f(k+3)=k-2$, then for $k>3$ we have the primitive subgraph
					$$\tilde{J}\rightarrow \hat{J}\rightarrow J_{r_{4}}\rightarrow\dots\rightarrow J_{r_{2k-2}}\rightarrow \tilde{J}$$
					of length $2k-2$ which leads to contradiction as in case (1b).
									\item If $f(k+4)=k-2$ and $f(k+3)=k \Rightarrow f(1)=k+1$, we have a second minimal $(2k+1)$ orbit given in \eqref{eq:validkk4} with topological structure max-min-max-min-max. 
				\item If $f(k+4)=k-2$ and $f(k+3)=k+1 \Rightarrow 4$-suborbit $\left \{ k-1, k+1, k+2, k+3 \right \}$, a contradiction.
			\end{enumerate}
		\end{enumerate}

Note that all four cyclic permutations are simple. Finally, we aim to analyze the digraphs and demonstrate that there are no primitive cycles of even length $\leq 2k-2$, which would imply by Straffin's lemma an existence of odd periodic orbit of length $\leq 2k-3$. From Lemma~\ref{thm:convStraffin} it then follows that the the $P$-linearization of the orbits \eqref{eq:validkk1}, \eqref{eq:validkk2}, \eqref{eq:validkk3}, and \eqref{eq:validkk4} present an example of continuous map with second minimal $(2k+1)$-orbit. The proof coincides with the similar proof given in Lemma~\ref{lem:settingij}. 		
			\end{proof}

		\begin{figure}%
		\centering
		\begin{tikzpicture}[
			mymat/.style={
			matrix of math nodes,
			text height=2.5ex,
			text depth=0.75ex,
			text width=18.25ex,
			align=center,
			column sep=-\pgflinewidth,
			row sep=-\pgflinewidth},
			nodes = draw,
			semithick
		 ]
										
			\matrix[ampersand replacement=\&,mymat,anchor=west,row 2/.style={nodes={draw,fill=gray!30}}] at (0,0) (ij)
			{
									 \&  				 \& (i-1, 2k-i+1) \& (i-1, 2k-i+2) \\
									 \& 	(i, 2k-i) \& (i, 2k-i+1)   \& \\
				(i+1, 2k-i-1) \& (i+1, 2k-i) \&  				   \& \\
			};
			\node (1) [below=0.2cm of ij-1-3, draw=none] {};
			\node (22) [above=0.2cm of ij-2-3, draw=none] {};
			
			\node (23) [left=0.6cm of ij-2-3, draw=none] {};
			\node (24) [right=0.3cm of ij-2-2, draw=none] {};
			
			\node(21) [below=0.2cm of ij-2-2, draw=none] {};
			\node(3) [above=0.2cm of ij-3-2, draw=none] {};
			
			\begin{scope}[shorten <= -2pt]
				\draw[->, blue] (ij-1-3.west) -- (ij-2-2.north);
				\draw[->, red] (ij-2-2.west) -- (ij-3-1.north);
				\draw[->, blue] (ij-1-4.south) -- (ij-2-3.east);
				\draw[->, red] (ij-2-3.south) -- (ij-3-2.east);
				
				\draw[<-, blue] (1) -- (22);
				\draw[<-, red] (23) -- (24);
				\draw[<-, red] (21) -- (3);
			\end{scope}
		\end{tikzpicture}	
		\caption{The cyclic permutation sharing mechanism for settings $(i,2k-i)$ and $(i,2k-i+1)$. Red arrows indicate cyclic permutations originating at a node and blue arrows indicate cyclic permutations shared from above.} 
		\label{fig:sharingMech}
	\end{figure}
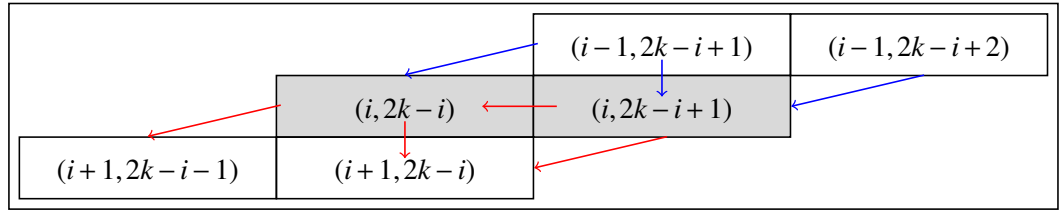	
	
\begin{lem}
		Let $\tilde{J}, \hat{J}$ be in setting $(i, 2k-i)$ for $2 < i < k$. For fixed $i$, this setting shares one cyclic permutation with the setting $(i-1, 2k-i+1)$ and another cyclic permutation with the setting $(i, 2k-i+1)$. When $i=k-1$, the setting $(k-1, k+1)$ shares a cyclic permutation with the case when $m=2k-1$ from Lemma~\ref{lem:m2km1}.
		\label{lem:sharingHoriz}
	\end{lem}
	\begin{proof}
		The proof is by direct comparison. Note that if $i>3$ the cyclic permutation \eqref{eq:validij4}  is transformed to \eqref{eq:validij1} after substitution $(i,j)$ with $(i-1,j+1)$. If $i=3$, \eqref{eq:validij1} repeats the cyclic permutation \eqref{eq:case2-2km2-4}. Therefore, the setting $(i, 2k-i)$ shares one cyclic permutation with $(i-1, 2k-i+1)$. We can also see that the cyclic permutations \eqref{eq:validij3} and \eqref{eq:validij12} are identical. So the setting $(i, 2k-i)$ also shares a cyclic permutation with the setting $(i, 2k-i+1)$. When $i=k-1$, the cyclic permutation \eqref{eq:validij4} and \eqref{eq:len2kcase11cyc} are identical.
		
		\end{proof}
	
	\begin{lem}
		Let $\tilde{J}, \hat{J}$ be in setting $(i, 2k-i+1)$ for $3\leq i \leq k$. For fixed $i$, this setting shares one cyclic permutation with the setting $(i-1, 2k-i+1)$ and another cyclic permutation with the setting $(i-1, 2k-i+2)$. 
		\label{lem:sharingVert}
	\end{lem}
	\begin{proof}
		The proof is once again by direct comparison. If $3 < i < k$, the substitution $(i,j)$ with $(i-1,j+1)$ in \eqref{eq:validij2} implies the cyclic permutation \eqref{eq:validij11} from Lemma~\ref{lem:settingijp1}. If $i=3$, \eqref{eq:validij11} repeats the cyclic permutation \eqref{eq:case2-2km2-3}, revealed in Lemma~\ref{lem:case2-2km2}. If $i=k$, choose $i=k-1$ in the cyclic permutation \eqref{eq:validij2} and observe that it is identical to \eqref{eq:validkk4}.This proves sharing with setting $(i-1, 2k-i+1)$. Similarly, if $i>3$ the substitution $(i,j)$ with $(i-1,j+1)$ cyclic permutation \eqref{eq:validij13} is transformed to the cyclic permutation \eqref{eq:validij14}. If $i=3$ \eqref{eq:validij14} repeats the cyclic permutation \eqref{eq:case2-2km1-3}, revealed in Lemma~\ref{lem:case2-2km1}. If $i=k$, then by choosing $i=k-1$ in \eqref{eq:validij13} we see that it is identical to \eqref{eq:validkk2}. This confirms sharing with the setting $(i-1, 2k-i+2)$. 
	\end{proof}

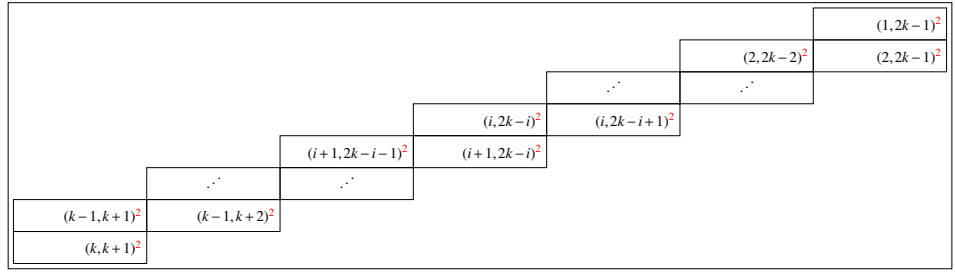
\begin{figure}[htb]	
	\centering				
	\resizebox{\textwidth}{!}{
	\begin{tikzpicture}[
			mymat/.style={
			matrix of math nodes,
			text height=2.5ex,
			text depth=0.75ex,
			text width=18.25ex,
			align=right,
			column sep=-\pgflinewidth,
			row sep=-\pgflinewidth},
			nodes = draw,
			vertex/.style={draw,circle}
		 ]
										
			\matrix[ampersand replacement=\&,mymat,anchor=west,row 3/.style={align=center},row 6/.style={align=center}] at (0,0) (ijgrid)
			{
				 \&  \&  \&  \&  			 \&					\& (1,2k-1)^{\color{Red} 2} \\
				 \&  \&  \&  \&  			 \& (2,2k-2)^{\color{Red} 2}	\& (2,2k-1)^{\color{Red} 2} \\
				 \&  \&  \&  \& \iddots \& \iddots 	\&  \\
				 \&  \&  \& (i,2k-i)^{\color{Red} 2} \& (i,2k-i+1)^{\color{Red} 2} \&  \&  \\
				 \&  \& (i+1, 2k-i-1)^{\color{Red} 2} \& (i+1, 2k-i)^{\color{Red} 2} \&  \&  \&  \\
				 \& \iddots \& \iddots \&  \&  \&  \&  \\
				(k-1, k+1)^{\color{Red} 2} \& (k-1,k+2)^{\color{Red} 2} \&  \&  \&  \&  \&  \\
				(k,k+1)^{\color{Red} 2} \&  \&  \&  \&  \&  \& \\
			};
			
		\end{tikzpicture}	} 		
	\caption{Demonstration of counting of distinct cyclic permutations per setting; here indicated with red numbers}
	\label{fig:sharingCount}
\end{figure}	

\begin{table}%
	\centering
	\begin{tabular}{r|c|r}\toprule
		Topological Structure & Count & Permutation \\ \midrule
		max & $1$ & \eqref{eq:case2-2km2-4} \\
		min-max & $1$ & \eqref{eq:case1-2km1-2} \\
		min-max-min & $1$ & \eqref{eq:case1-2km1-1} \\
		max-min & $2$ & \eqref{eq:case2-2km1-2},\eqref{eq:case2-2km1-3} \\
		max-min-max & $2k-3$ & \eqref{eq:len2kcase22cyc},\eqref{eq:case2-2km2-3},\eqref{eq:validkk3},({\color{Green}\eqref{eq:validij1}},{\color{Green}\eqref{eq:validij4}}),({\color{Green}\eqref{eq:validij13}},{\color{Green}\eqref{eq:validij14}}) \\
		max-min-max-min-max & $2k-5$ & \eqref{eq:validkk1}, ({\color{Green}\eqref{eq:validij2}},{\color{Green}\eqref{eq:validij11}}),({\color{Green}\eqref{eq:validij3}},{\color{Green}\eqref{eq:validij12}}) \\ \bottomrule
	\end{tabular}
\caption{Counts for topological structure of second minimal odd periodic orbits. {\color{Green}Green} entries correspond to permutations in settings $(i,2k-i),(i,2k-i+1)$ for $2<i<k$ with $k>3$. Parentheses indicate permutations that may be shared.}
\label{tab:topStruct}
\end{table}
The sharing mechanism provided in Lemma~\ref{lem:sharingHoriz} and Lemma~\ref{lem:sharingVert} is illustrated in \ref{fig:sharingMech}. Each setting $(i,2k-i)$ and $(i,2k-i+1), 2<i<k$, $(k,k+1)$ contain exactly $4$ second minimal cyclic permutations, which  are shared with neighboring settings. In particular, for setting $(i,2k-i)$, two cyclic permutations are inherited from its two neighbors immediately to the right, and other two are shared with neighbors immediately down. Observe, that we have also demonstrated the sharing extending to the cases $(1,2k-1)$, $(2,2k-1)$, $(2, 2k-2)$, and $(k,k+1)$. To count all the different second minimal cyclic permutations we start in the upper right corner $(1,2k-1)$ of the table in Fig.~\ref{fig:sharingCount} and work our way down to the bottom left corner $(k,k+1)$ by successively moving down and left. Due to sharing mechanism, the number of new second minimal cyclic permutations produced in each setting is equal to 2 (written as a superscript to the setting). There are $k-1$ columns, each with $4$ distinct cycles giving a total of $4(k-1)$ distinct cyclic permutations. Adding the $1$ remaining permutation from Lemma \ref{lem:m2km1} we have the required number, $4k-3$, of second minimal cyclic permutations of period $2k+1$, unique up to an inverse. They are all simple positive type according to the Definition~\ref{en:simpletypea}. All the inverse cyclic permutations of the constructed $4k-3$ orbits constitute all simple negative type second minimal $2k+1$-orbits. Finally, we count the different types of topological structure of all the $4k-3$ positive type second minimal orbits and present the results in Table~\ref{tab:topStruct}. The inverse cyclic permutations have the same topological structures with "max" and "min" exchanged.

\end{proof}

\bibliographystyle{plain}	
	\bibliography{GeneralSecondMinimalProofs}

\end{document}